\documentclass[aos,preprint]{imsart}

\usepackage{amsthm,amsmath}
\usepackage{graphicx,amssymb,bbm}
\RequirePackage[colorlinks,citecolor=blue,urlcolor=blue]{hyperref}
\RequirePackage{natbib}

\startlocaldefs

\newcommand{\Ind}{\mathbbm{1}}
\DeclareMathOperator{\tr}{Tr}
\newcommand{\Ac}{\mathcal{A}}
\newcommand{\Ic}{\mathcal{I}}
\newcommand{\Vc}{\mathcal{V}}
\newcommand{\eps}{\varepsilon}
\newcommand{\floor}[1]{\left\lfloor#1\right\rfloor}
\renewcommand{\Pr}{\mathbb{P}}
\newcommand{\E}{\mathbb{E}}
\newcommand{\R}{\ensuremath{\mathbb{R}}}
\newcommand{\dlim}{\stackrel{d}{\rightarrow}}
\DeclareMathOperator{\Tr}{Tr}
\DeclareMathOperator{\argmax}{argmax}
\newcommand{\iid}{\stackrel{\text{i.i.d.}}{\sim}}

\renewcommand{\ij}[1]{i_{#1}j_{#1}}
\newcommand{\allij}{i_1j_1,\dots,i_kj_k}

\DeclareMathOperator{\Exp}{Exp}

\newtheorem{lemma}{Lemma}

\newtheorem{thm}{Theorem}
\newtheorem{conj}{Conjecture}
\endlocaldefs

\begin{document}

\begin{frontmatter}

\title{Adaptive testing for the graphical lasso}
\runtitle{Adaptive testing for the graphical lasso}

\author{\fnms{Max} \snm{Grazier G'Sell}\corref{}\ead[label=e1]{maxg@stanford.edu}\thanksref{t1}},
\author{\fnms{Jonathan} \snm{Taylor}\ead[label=e2]{jonathan.taylor@stanford.edu}\thanksref{t2}}
\and
\author{\fnms{Robert} \snm{Tibshirani}\ead[label=e3]{tibs@stanford.edu}\thanksref{t3}}

\thankstext{t1}{Supported by a NSF GRFP Fellowship.}
\thankstext{t2}{Supported by NSF grant DMS 1208857 and AFOSR grant 113039.}
\thankstext{t3}{Supported by NSF grant DMS-9971405 and NIH grant N01-HV-28183.}

\address{Department of Statistics\\
Stanford University\\
Sequoia Hall\\
Stanford, California 94305-4065\\
\printead{e1}\\
\phantom{E-mail:\ }\printead*{e2}\\
\phantom{E-mail:\ }\printead*{e3}}

\affiliation{Stanford University}

\runauthor{M. Grazier G'Sell et al.}

\begin{abstract}
  We consider tests of significance in the setting of the graphical lasso for
  inverse covariance matrix estimation.  We propose a simple test statistic
  based on a subsequence of the knots in the graphical lasso path.  We show
  that this statistic has an exponential asymptotic null distribution, under
  the null hypothesis that the model contains the true connected components.

  Though the null distribution is asymptotic, we show through simulation that
  it provides a close approximation to the true distribution at reasonable
  sample sizes.  Thus the test provides a simple, tractable test for the
  significance of new edges as they are introduced into the model.  Finally, we
  show connections between our results and other results for regularized
  regression, as well as extensions of our results to other correlation matrix
  based methods like single-linkage clustering.
\end{abstract}



\end{frontmatter}


\section{Introduction}\label{section:intro}

We consider the problem of hypothesis testing for the components of the inverse
covariance matrix.  We focus on the particular case of the graphical lasso
\citep{glasso}, and construct a set of interpretable hypotheses tests and
corresponding test statistics.

In the graphical lasso, the inverse covariance matrix
$\Sigma^{-1}$ is estimated by the matrix $\Theta$ that maximizes
\begin{align*}
  \log\det\Theta -\tr(S\Theta)-\rho||\Theta||_1,
\end{align*}
where $S$ is the sample covariance matrix of the data, $\rho >0$, and $||\Theta||_1 =
\sum_{i\ne j} |\Theta_{ij}|$, the
element-wise $L_1$-norm excluding the diagonal.  We restrict ourselves to the
particular case where the data have been centered and scaled so that $S$ is the sample correlation
matrix.

As the regularization parameter $\rho$ decreases, the resulting estimate
$\Theta$ is correspondingly denser.  Solutions can be computed over a range of
$\rho$ values using existing approaches \citep[e.g.][]{glasso, witten2011new,
rahulpaper, rahulpaper2}.  Some theoretical guarantees also exist 
\citep[e.g.][]{ravikumar2008model} about the recovery of the appropriate sparsity
patterns in $\Theta$  under certain conditions.  However, there has not been
much work in inference and testing in the graphical lasso setting.

A similar problem of testing for the lasso in linear regression was recently
addressed by \cite{lockhart2013significance}.  That paper constructs a
covariance test statistic, based on the change in covariance between fitted values of the
lasso and the original observed values, at points along the lasso path where
new variables enter.  Intuitively, this statistic should be large if the
variables entering carry information.  In the case of orthogonal predictors, this statistic simplifies to a
knot-based statistic, $n\lambda_k(\lambda_k-\lambda_{k+1})$, based on the
spacing between points in the lasso path where the set of selected variables
change.  In their paper, they show that these statistics have asymptotic
$\Exp(1/k)$ distributions for the $k^{th}$ null variable to enter the lasso estimate.

Inspired by their work, this paper proposes a set of hypotheses
corresponding to a particular sequence of $\rho$
along the graphical lasso solution path.  These hypotheses test the
statistical significance of structure added after each of these points in the
graphical lasso solution set.  For these hypotheses, we construct
simple test statistics, and prove that they have an exponential asymptotic null
distribution.  We also demonstrate through simulation that these distributions
are practically useful at reasonable sample sizes.

Finally, our test statistics and the corresponding theory have close ties to
particular functions of the order statistics of the correlation matrix.  As a
result, we are able to extend our results to other correlation-based methods.
In particular, for the case of single-linkage clustering based on absolute
correlations, we are able to provide similar test statistics and asymptotic
null distributions.

Section \ref{glassotest} defines the proposed test statistic and gives an overview of its
behavior, including a demonstration in Section \ref{glassotest:example} of this
test statistic on the cell-signaling data set from \cite{glasso}.  Section \ref{section:details}
discusses the behavior of the proposed test statistic and derives its
exponential asymptotic null distributions.  Section \ref{section:simulations} demonstrates
empirically that this null distribution holds in a variety of settings and
demonstrates desirable behavior at later steps that is yet unproven.  We conclude in Section
\ref{section:discussion} with a discussion of the behavior and interpretation
of this statistic, along with its extension to correlation-based clustering.

\section{Significance testing in the graphical lasso}\label{glassotest}

Estimation based on the graphical lasso is interesting, because it allows
estimation of sparse dependence structures of data.  As the regularization
parameter $\rho$ varies, so does the level of sparsity.  This gives rise to an
inferential question: how much structure can be justified in the presence of
noise?

In this section, we define a simple set of hypotheses at a set of parameter
values where the sparsity pattern changes dramatically.  We characterize the
behavior of the estimator at these points, and its relationship to order
statistics of the correlation matrix.  The estimate at these points resembles
the situation that arises in \cite{lockhart2013significance} for the lasso
estimator in linear regression.  

In \cite{lockhart2013significance}, the authors propose a covariance-based test
statistic, which is a function of the lasso fitted values at the points in the lasso path where variables
first enter.  They show that this statistic has a simple $\Exp(1)$ asymptotic null
distribution, under the null that the true signal variables have already been
selected.

In this section, we propose a similar test statistic for our setting and the
hypotheses mentioned above.  We also provide an example demonstrating this test
statistic on biological data. In later parts of this paper, we prove that this proposed statistic has similar desirable properties to the
original lasso statistic of \cite{lockhart2013significance}. 

\subsection{The graphical lasso path}\label{glassotest:path}

We begin by summarizing some known properties of the graphical lasso estimator.
We refer to the graphical lasso \emph{path} to be the set of solutions
$\Theta(\rho)$ over all values of $\rho$.  Broadly speaking, $\Theta(\rho)$
will be sparse for large values of $\rho$ and dense for small values of $\rho$.

More explicitly, there exists a finite set of graphical lasso \emph{knots},
$\rho_1 = \max_{i\ne j} |S_{ij}| \ge \rho_2 \ge \rho_3 \ge \cdots \ge \min_{i\ne j}|S_{ij}| \ge 0$,
which are the points where the sparsity pattern in the estimate
$\Theta(\rho)$ changes.  Above $\rho_1$, $\Theta(\rho)$ is diagonal; below the
smallest knot, $\Theta(\rho)$ is dense.

\begin{figure}[ht!]
  \centering
  \includegraphics[width=3in]{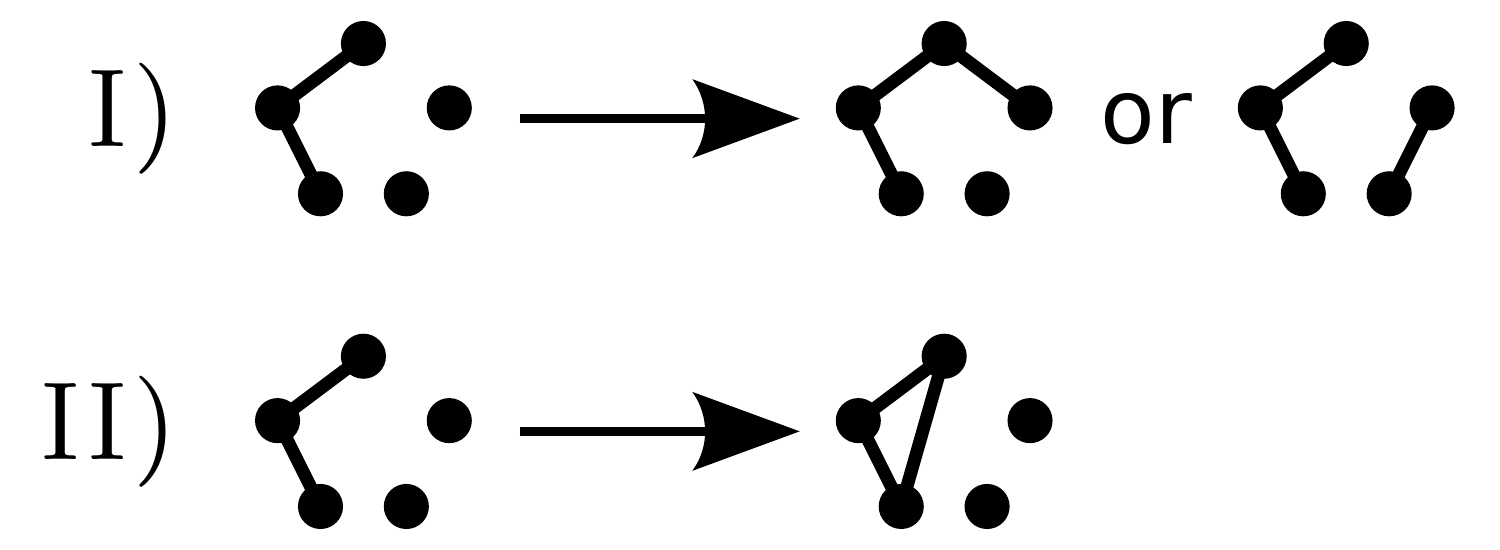}
  \caption{Illustration of the two types of sparsity pattern changes in the
  graphical lasso.  In case I, an edge is formed between two previously
  disconnected components.  In case II, an edge is formed within a previously
  connected component.  This paper addresses the knots at which the first type
  of change occurs.} 
  \label{fig:knotcase}
\end{figure}

If we consider the estimate as we move from larger parameter values to smaller
ones, one of two changes to the sparsity pattern can take place at a knot
$\rho_k$.  These are illustrated in Figure \ref{fig:knotcase}.  First, two groups of variables that were disconnected for all $\rho > \rho_k$
can become connected by a new edge.  Second, the sparsity pattern within a
previously connected group of variables can change, without altering the
overall set of connected components of variables. 

The first of these changes is particularly interesting.  These points have been 
characterized by \cite{witten2011new} and \cite{rahulpaper}, where they are
used to decompose the graphical lasso problem into smaller problems on the
connected components.  We are interested in them because they give rise to
natural hypotheses to test, and because their special behavior will lead those
hypotheses to be naturally nested.

Let $\tilde{\rho}_1 \ge \cdots \ge \tilde{\rho}_M$ correspond to the subset of
knots where the connected components of $\Theta{\rho}$ change.  As discussed in
detail in the papers referenced above, the $\Theta_{ij}(\rho) \ne 0$ for all
$\rho < |S_{ij}|$, and in particular, disconnected components can only become
connected through edges that enter in such a fashion.  As a result, the $\tilde{\rho}_k$ correspond to
points where $\rho = |S_{ij}|$ for some $i\ne j$.  Furthermore,
$\tilde{\rho}_1,\tilde{\rho}_2,\dots$ corresponds to the subsequence of order
statistics of $|S_{ij}|$ obtained by ordering the $|S_{ij}|$ from largest to
smallest, and then eliminating those $|S_{ij}|$ where $i$ and $j$ would already
be in the same connected component based on the preceding elements.

This allows the points $\tilde{\rho}_1,\dots,\tilde{\rho}_M$ of the graphical
lasso path to be characterized in terms of the extreme
values of the off-diagonal elements of $S$.  This provides inspiration for the
test statistics in the next section.  It also forms the basis for the theory
used in Section \ref{section:details} to prove the asymptotic behavior of those test
statistics.

\subsection{Testing along the graphical lasso path}
\label{glassotest:test}

With $\tilde{\rho}_1\ge \cdots \ge \tilde{\rho}_M$ defined to be the points
where the connected components of the graphical lasso estimator $\Theta(\rho)$
change, we can define a corresponding sequence of hypotheses $H_1,\dots,H_M$.
\emph{The hypothesis $H_k$ is that each connected component of the
true $\Sigma^{-1}$ is contained within a connected component of $\Theta(\rho)$, $\forall \rho < \tilde{\rho}_k$.}

Because the nature of the graphical lasso path is such that connected
components can never become disconnected, these hypotheses are nested by
construction.  That is, if $H_k$ is true, then $H_{\ell}$ is true for all $\ell
> k$.  These hypotheses correspond to natural questions about the dependence graph.
They describe whether we have found all the important connections between
groups of variables, where groups of variables can include singletons.

To test these hypotheses, we note that the $\tilde{\rho}_k$ resemble the 
regularization parameters of the lasso estimator in the case of linear
regression with orthogonal $X$, including the relationship between the
knot location and the ordered absolute values of the underlying data.  The
recent paper \citep{lockhart2013significance} constructs a test statistic for
similar hypotheses in that setting; their work provides inspiration for the
test statistics we propose here.

We propose the test statistic 
\begin{align*}
  T_k = n\tilde{\rho}_k(\tilde{\rho}_k - \tilde{\rho}_{k+1}).
\end{align*}
Intuitively, this statistic will tend to be large when signal is present, as the
$\tilde{\rho}$ values converge to large correlations and the $n$ grows.  In the
absence of signal, we will show that the spacings $\tilde{\rho}_k -
\tilde{\rho}_{k+1}$ have simple limiting distributions when scaled by
$n\tilde{\rho}_k$.

The statistic $T_k$ is similar to the knot-form of the statistic
discussed by \cite{lockhart2013significance} and summarized here in Section
\ref{section:intro}; we discuss the relationship
between $T_k$ and the more general covariance statistic form of
\cite{lockhart2013significance} in Section
\ref{section:define:equiv}.

We will show in Section \ref{section:details} that the $T_k$ have a simple
asymptotic null distribution, under the hypotheses $H_k$.  In particular, let
$k$ be the smallest value such that $H_k$ is true.  Then $T_k \dlim \Exp(1)$, and
$T_{k'} \dlim \Exp(1/(k'-k+1))$ for $k' > k$.  We also demonstrate empirically
that these asymptotic distributions are accurate even for reasonable sample
sizes and common signal structures.  This, combined with the simplicity of the
null distributions, makes the proposed statistics a reasonable candidate for directly testing the
significance of elements in the graphical lasso solution.

\subsection{Example}\label{glassotest:example}
Before proceeding with theoretical results, we demonstrate the proposed
statistic on the cell signaling data from
\cite{glasso}.  This data consists of a set of 7466 observations of 11
proteins.  We apply the graphical lasso to estimate to estimate a sparse dependence graph
describing the relationships between these proteins.

This experimental data set contains only a few variables, and was designed to
contain strong correlations.  Because our proposed statistic is really
interesting in the presence of a mix of informative and uninformative
variables, we augment the real data set with 100 noise variables, drawn from a Gaussian
distribution with zero correlation between the variables.  We then fit the
graphical lasso on this augmented data set.  With this construction, we expect
the true connected components to involve the true proteins, and none of the
noise variables.

We apply the tests from Section \ref{glassotest:test} to obtain test statistics
and p-values for the
first fifteen edges as they enter.  This is performed on repeated subsamples of
$n=500$ observations to capture variation in the resulting statistics and
$p$-values.  The $p$-values are computed in comparison to an $\Exp(1)$
distribution, which should be accurate for the first null step and conservative
for later null steps.  The results are shown in Figure \ref{fig:dataexample}.

\begin{figure}[ht!]
  \centering
  \includegraphics[width=5in]{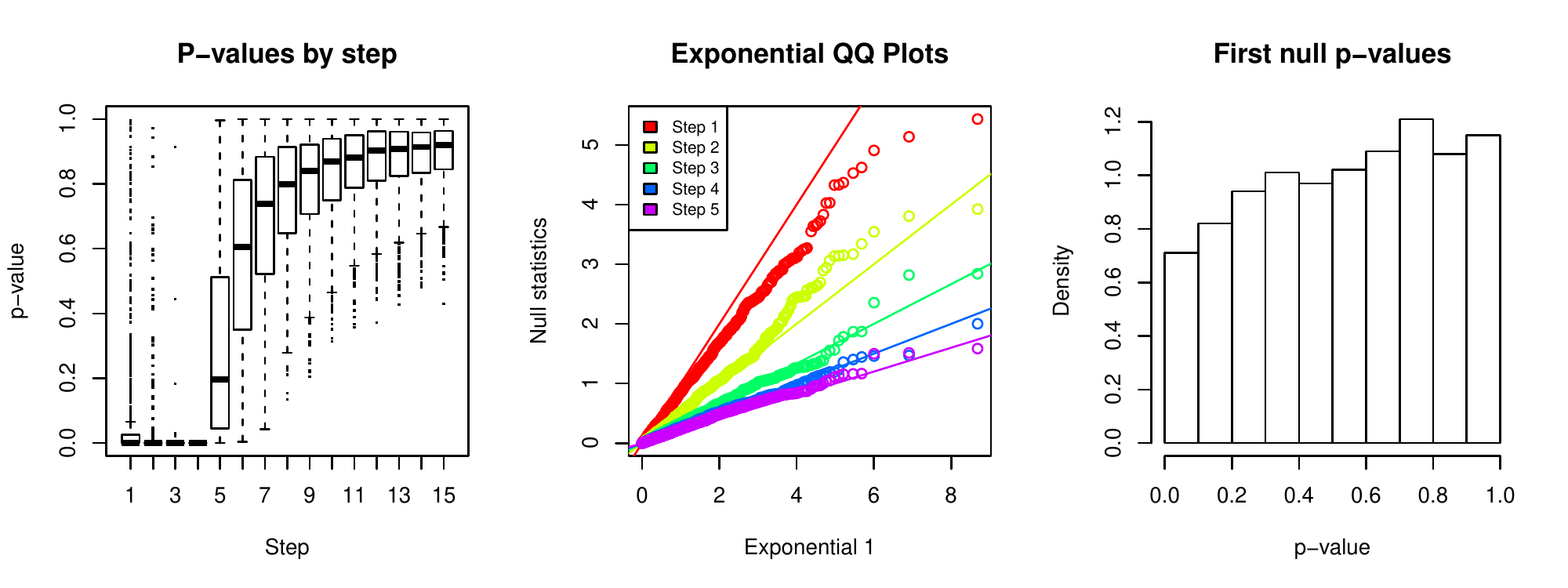}
  \caption{The proposed test statistics and their $p$-values are shown,
  using data consisting of 11 related biological proteins and 100 simulated
  noise variables on repeated subsamples of size $n=500$.  The first four to
  five steps, depending on the realization,
  correspond to the entrance of edges between true protein
  variables, with the next step corresponding to a connection to a noise
  variable.  The left panel shows boxplots of the $p$-values, where each
  statistic is compared to the CDF of an $\Exp(1)$ distribution.  Note that
  this will be conservative after the first null step.  The middle panel shows
  $\Exp(1)$ QQ-plots for the first five null steps, so that a line with slope
  $1/k$ should correspond to an Exponential distribution with mean $1/k$.  The right plot
  shows a histogram of the $p$-value at the first null step. } 
  \label{fig:dataexample}
\end{figure}

With this sample size, the first four to five edges to enter (depending on the
particular realization) correspond to edges in the original 11 signal
variables.  The next edge to enter is a noise edge, connecting either two noise
variables or a real variable and a noise variable.  The fact that the location
of this step is realization-dependent makes it difficult to accurately assess
the distribution of the first null step in the left plot.

The center plot shows QQ-plots for the distributions of the first five noise
edges accepted, wherever they occur in a particular realization.  This shows
that the distributions are close to $\Exp(1/k)$ distributions.  The right plot
shows a histogram of the $p$-values for the first of these noise edges, where
they appear close to uniform.  We attribute the conservative skew that can be seen in both
the center and right plots to the presence of signal edges that have not yet
been selected because their strength is on the same scale as the noise edges.
In the more idealized simulations of Section \ref{section:simulations}, we see that this
conservative trend vanishes.

\subsection{Covariance Formulation}\label{section:define:equiv}

In \cite{lockhart2013significance}, the knot-based statistic we discussed in
Section \ref{glassotest:test} was a special case of a more general covariance
statistic, which applied more broadly.  The covariance test statistic
corresponded to a difference in inner products between the observed and fitted
values.  The equivalent formulation in the context of the graphical lasso would be
\begin{align*}
  \tilde{T}_k =
  \frac{n}{2}\left(\Tr\left(S\Sigma(\tilde{\rho}_{k+1})\right)-\Tr\left(S\Sigma_0(\tilde{\rho}_{k+1})\right)\right)
\end{align*}
where $\Sigma_0(\tilde{\rho}_{k+1})$ is the graphical lasso solution
$\Sigma(\tilde{\rho}_{k+1})$, but without the change in nonzero pattern that
would have happened at $\tilde{\rho}_k$, so that any zero elements before
$\tilde{\rho}_k$ remain zero.

Let $i^*,j^*$ be the element that entered at $\tilde{\rho}_k$.  Since $\Sigma_0(\rho_{k+1}$ is identical to
$\Sigma(\rho_{k+1})$ except at $i^*,j^*$, the expression simplifies to 
\begin{align*}
  \tilde{T}_k &= nS_{i^*j^*}\left(S_{i^*j^*}\pm \tilde{\rho}_{k+1}\right) =
  n\tilde{\rho}_k\left(\tilde{\rho}_k - \tilde{\rho}_{k+1}\right) = T_k,
\end{align*}
so the covariance form of the statistic and the knot form of the statistic are
the same for our problem.

It is interesting to note that if used the original graphical lasso knots
$\rho_k$ instead of the restricted sequence $\tilde{\rho}_k$ corresponding to
changes in connected components, this connection between $T_k$ and $\tilde{T}_k$
fails to hold.

\section{Null behavior}\label{section:details}
Here we give justification and theory for the behavior of the statistic proposed in
Section \ref{glassotest:test} in the null case.  We begin with the global null
 of no correlation structure in Section \ref{section:details:globalnull}, and
 then extend these ideas in Section \ref{section:details:aftersignal} to a
weaker null case where correlation structure exists but has been selected.

\subsection{Global null, first step}\label{section:details:globalnull}

We consider the test statistic, $T_1 = n\tilde{\rho}_1(\tilde{\rho}_1-\tilde{\rho}_2)$, for the first edge to enter on the
path, under the global null hypothesis that the data has no correlation
structure ($H_1$ from Section \ref{glassotest:test}).  Explicitly, this is the
hypothesis that $\Sigma = I_p$, where $\Sigma$ is
the true correlation matrix and $I_p$ is the $p\times p$ identity matrix.  We
discuss this case first, since the proof is simpler and yet contains all of the
important pieces for the more general case.

The test statistic is most easily described in terms of the order statistics
of the correlation matrix.  Let $V_1,V_2$ be the first and second largest
off-diagonal elements of $S$ (in absolute value).  Since the first two edges
selected necessarily connect previously disconnected components, we have
$\tilde{\rho}_1 = V_1$ and $\tilde{\rho}_2 = V_2$, so $T_1 = n V_1(V_1-V_2)$. 

Under the null distribution, $V_1$ and $V_2$ are the largest and second largest correlations out
of $p$ i.i.d. spherically distributed vectors in $\R^n$.  Distributions of
these correlations and their extrema are discussed in \cite{corrlimit}.  We show
that as $n,p\to \infty$ with $n$ growing faster than $\log p$, $T_1$ converges
in distribution to $\Exp(1)$. 

\begin{thm}\label{thm:globalfirststep} Let $V_1\ge V_2$ be the first two order
  statistics of the absolute values of the correlations of
  $Z_1,\dots,Z_p\in\R^n$, where the $Z_i$ are i.i.d. with a spherical
  distribution around 0.  Then $nV_1(V_1-V_2) \dlim \Exp(1)$ as $n,p\to\infty$
  with $\frac{\log p}{n} \to 0$.  \end{thm} \begin{proof} Begin by defining the
    following quantities: \begin{enumerate} \item $M_{ij} = \max_{(i',j')\ne
        (i,j)} |S_{i'j'}|$ and $M_{ij}^* = \max_{(i',j')\notin \{i,j\}}
        |S_{i'j'}|$.  Note
        that $M_{ij}^* \le M_{ij}$.  \item Events $A_{ij} = \{|S_{ij}| >
        M_{ij}\}$ and $A_{ij}^* = \{|S_{ij}| > M_{ij}^*\}$.  Note that
        $\{|S_{ij}|(|S_{ij}|-M_{ij})>t/n\} \subseteq A_{ij}$,
        $\{|S_{ij}|(|S_{ij}|-M_{ij}^*)>t/n\} \subseteq A_{ij}^*$, that the $A_{ij}$
        are disjoint, and that $A_{ij}\subseteq A_{ij}^*$.  The events
        $A_{ij}$, $i<j$,
        form a partition, since exactly one $|S_{ij}|$ is larger than all the
        others.  The events $A_{ij}^*$ form an approximate partition; Lemma
        \ref{lemma:astar} shows that the approximation is close.  \item
        $(i^*,j^*) = \argmax_{i<j} |S_{ij}|$, the indices of the largest
        off-diagonal element of $S$.  \item $E = \{M_{i^*j^*} =
        M_{i^*j^*}^*\}$. This is the event that the second largest element of
        $S$ shares no indices with the largest element.  In particular, this is
        a subset of the event in Lemma \ref{lemma:maximadisconnect}, so
        $\Pr(E^c) \to 0$ as $n,p\to \infty$ with $\frac{\log p}{n}\to 0$.
    \end{enumerate}

  We can expand $\Pr(T_1 > t) = \Pr(nV_1(V_1-V_2)>t)$ as
  \begin{align*}
    \Pr(nV_1(V_1-V_2)>t) &= \Pr(V_1(V_1-V_2) > t/n) = \sum_{i<j}^p
    \Pr\left(|S_{ij}|\left(|S_{ij}|-M_{ij}\right)>t/n\right).
  \end{align*}
  Here we are taking advantage of the fact that
  $\Pr\left(|S_{ij}|\left(|S_{ij}|-M_{ij}\right)>t/n\right)$ is only positive for one
  pair $i<j$, so the events being considered are disjoint and cover the event
  that $V_1(V_1-V_2)>t/n$.

  This sum can be further expanded based on intersection with $E$, obtaining  
  \begin{align*}
    \Pr(nV_1(V_1-V_2)>t) &= \sum_{i<j}^p
    \Pr\left(\left\{|S_{ij}|\left(|S_{ij}|-M_{ij}\right)>t/n\right\}\cap
    E\right)\\
    &\qquad + 
    \sum_{i<j}^p \Pr\left(\left\{|S_{ij}|\left(|S_{ij}|-M_{ij}\right)>t/n\right\}\cap
    E^c\right).
  \end{align*}
  Note that
  $\Pr\left(\left\{|S_{ij}|\left(|S_{ij}|-M_{ij}\right)>t/n\right\}\cap
  E^c\right) \le \Pr(A_{ij}\cap E^c)$, so
  \begin{align*}
    &\left|\Pr(nV_1(V_1-V_2)>t) - \sum_{i<j}^p
    \Pr\left(\left\{|S_{ij}|\left(|S_{ij}|-M_{ij}\right)>t/n\right\}\cap
    E\right)\right|\\
    &\quad = \sum_{i<j} \Pr\left(\left\{|S_{ij}|\left(|S_{ij}|-M_{ij}\right)>t/n\right\}\cap
    E^c\right)
     \le \sum_{i<j} \Pr(A_{ij}\cap E^c) = \Pr(E^c)
  \end{align*}

  By a similar argument, 
  \begin{align*}
    &\left|\sum_{i<j}^p
    \Pr\left(\left\{|S_{ij}|\left(|S_{ij}|-M^*_{ij}\right)>t/n\right\}\right) -
    \sum_{i<j}^p \Pr\left(\left\{|S_{ij}|\left(|S_{ij}|-M^*_{ij}\right)>t/n\right\}\cap
    E\right)\right|\\
    &\quad = \sum_{i<j} \Pr\left(\left\{|S_{ij}|\left(|S_{ij}|-M^*_{ij}\right)>t/n\right\}\cap
    E^c\right)\\
    &\quad \le \sum_{i<j} \Pr\left(A_{ij}^*\cap E^c\right) = \sum_{i<j}
    \Pr(A_{ij}\cap E^c) + \sum_{i<j}\Pr\left( (A_{ij}^*\backslash A_{ij})\cap E^c\right)\\
    &\quad \le \Pr(E^c) + \sum_{i<j} \Pr(A_{ij}^*\backslash A_{ij})
  \end{align*}

  Noting that $\Pr(\{|S_{ij}|(|S_{ij}|-M_{ij})>t/n\}\cap
  E)=\Pr(\{|S_{ij}|(|S_{ij}|-M^*_{ij})>t/n\}\cap E)$, we can combine these two
  statements to obtain
  \begin{align*}
    \left|\Pr(nV_1(V_1-V_2)>t) - \sum_{i<j}^p
    \Pr\left(\left\{|S_{ij}|\left(|S_{ij}|-M^*_{ij}\right)>t/n\right\}\right)\right|\le
    2\Pr(E^c) + \sum_{i<j} \Pr(A_{ij}^*\backslash A_{ij}).
  \end{align*}
  
  We apply Lemma \ref{lemma:gaplimit} to $M_{ij}^*$ (which is now
  $(p-2)\times(p-2)$) to approximate the second term, giving 
  \begin{align*}
    &\left|\Pr(nV_1(V_1-V_2)>t) - e^{-t}\left(\sum_{i<j}^p
    \Pr\left(A_{ij}^*\right)\right)\left(1+\frac{1}{\sqrt{\log p}}+\frac{\log
    p}{n}\right)\right|\\
    &\qquad\le
    2\Pr(E^c) + \sum_{i<j} \Pr(A_{ij}^*\backslash A_{ij}) +
    O\left(p^2e^{-(p-2)^{3/5}/4\sqrt{\log (p-2)}}\right).
  \end{align*}
  By Lemma \ref{lemma:astar}, $\sum_{i<j}^p \Pr\left(A_{ij}^*\right) \to 1$,
  and since $\sum_{i<j} \Pr(A_{ij}) = 1$, $\sum_{i<j} \Pr(A_{ij}^*\backslash
  A_{ij})$, this implies that 
  \begin{align*}
    \Pr(T_1 > t) &= \Pr(nV_1(V_1-V_2)>t) \to e^{-t}
  \end{align*}
  and therefore $T_1 \dlim \Exp(1)$.
\end{proof}

Here, and throughout this paper, we use the notation $f(x) = O(g(x))$ to denote
the situation where $\exists M\in \R^+$, $x_0\in\R$ such that $|f(x)| \le
M|g(x)|$ for all $x>x_0$.

This result establishes that under the global null hypothesis, $T_1 \dlim
\Exp(1)$.  The proofs of the supporting Lemmas are left for the Appendix.  Note that throughout this paper, we define $\Exp(\mu)$ to refer to the exponential
distribution with mean $\mu$ (not rate $\mu$).

\subsection{Global null, later steps}

The results of \cite{lockhart2013significance} for the lasso suggest that it
should also be true that $T_k = n\tilde{\rho}_k\left(\tilde{\rho}_k -
\tilde{\rho}_{k+1}\right) \dlim \Exp(1/k)$ under the same global null.  This
is supported by simulation results like those shown in Figure
\ref{fig:globalnullsim} and \ref{fig:statmeans}, which empirically observe that
these distributions are nearly exponential, and that the means agree with the
$1/k$ prediction.  

The proof that $n\tilde{\rho}_k\left(\tilde{\rho}_k -
\tilde{\rho}_{k+1}\right) \dlim \Exp(1/k)$ given in this section relies on a
conjecture, which we have so far been unable to prove.  This conjecture allows
the approximation from \cite{corrlimit} to be applied when bounding the
probability that $k$ independent correlations exceed the maximum correlation of
a large correlation matrix.  The conjecture is as follows.

\begin{conj}\label{conjecture}
  Let $f_n(x)$ be the density of $\sqrt{n}|S_{ij}|$, $\bar{F}_n(x) =
  \int_x^{\sqrt{n}} f_n(x)$, and $G_{n,p}(x)$ be the CDF of the maximum of an
  independent $p\times p$ correlation matrix (based on $n$ observed $p$
  vectors) scaled by $\sqrt{n}$.  Then for fixed $k$, 
  \begin{align*}
    \int_0^{\sqrt{\left(4-\frac{2}{k+2}\right)\log p}} G_{n,p}(x) \bar{F}_n^{k-1}(x)f_n(x)dx =
    o\left(\frac{1}{p^{2k}}\right).
  \end{align*}
\end{conj}

A stronger, but sufficient condition for this conjecture to hold is that
$p^{2k}\Pr\left(\sqrt{n}M_{n,p} < \sqrt{\left(4-\frac{2}{k+2}\right)\log p}\right)
\to 0$ for fixed $k$.

We believe that this conjecture holds.  The conjecture, and the stronger
sufficient condition, would hold if the elements of $\sqrt{n}S$ were independent
Gaussians.  As $n$ and $p$ grow large, the elements of $\sqrt{n}S$ are very
nearly independent Gaussians; the rest of the proofs in this paper
take advantage of that near-Gaussianity.

Assuming Conjecture \ref{conjecture} holds, the following theorem proves that
$T_k = n\tilde{\rho}_k(\tilde{\rho}_k-\tilde{\rho}_{k+1}) \dlim \Exp(1/k)$. 

As
in the previous theorem, we find it convenient to work with the order
statistics of the absolute correlation matrix, $V_1,V_2,\dots$.  On the event
$E_d$ that the first $d+1$ largest elements of $S$ share no indices, an event
which is defined below and shown to satisfy $\Pr(E_d) \to 1$, $\tilde{\rho}_k =
V_k$ for $k\le d$, so the results hold equally well for the
$\tilde{\rho}_1,\tilde{\rho}_2,\dots$ sequence of interest.

\begin{thm}\label{thm:globallatersteps}
  Let $V_1\ge V_2 \ge \cdots \ge V_{p(p-1)/2}$ be the order statistics of the
  absolute values of the correlations of $Z_1,\dots,Z_p\in\R^n$, where the $Z_i$ are
  i.i.d. with spherical distribution around 0.  Assume that Conjecture
  \ref{conjecture} holds.  Then for fixed $d\ge 0$ and $1
  \le k \le d$, $nV_k(V_k-V_{k+1})
  \dlim \Exp(1/k)$ as $n,p\to\infty$ with $\frac{\log p}{n} \to 0$.
\end{thm}
\begin{proof}
This proof follows nearly identically to the proof of Theorem
\ref{thm:globalfirststep}, using more general forms of the same results.  

Consider $T_k = nV_k(V_k-V_{k+1})$.  We can expand $\Pr(T_k > t)$ similarly to
the previous proof, but with $k$ greater than 1.  To do this, we need
a more complicated index set.  Define
\begin{align*}
  \Ic_k &= \left\{i_1,\dots,i_k,j_1,\dots,j_k \in \{1,\cdots,p\}: i_{\ell} <
  j_{\ell}\text{ and } (i_{\ell},j_{\ell}) \notin
  \bigcup_{\ell'<\ell}\{(i_{\ell'},j_{\ell'})\}, \forall \ell\right\}\\
  \Ic_k^* &= \left\{i_1,\dots,i_k,j_1,\dots,j_k \in \{1,\cdots,p\}: i_{\ell} <
  j_{\ell}\text{ and } i_{\ell},j_{\ell} \notin
  \bigcup_{\ell'<\ell}\{i_{\ell'},j_{\ell'}\}, \forall \ell\right\},
\end{align*}
so that $\Ic_k$ is the set of all selections of $k$ off-diagonal elements, and $\Ic_k^*$ is the set of all selections of $k$ off-diagonal elements
such that none share an index.

Then
\begin{align*}
  \Pr(T_k>t) &= \sum_{\Ic_k} \Pr\left(\{|S_{\ij{k}}|(S_{\ij{k}}-M_{\allij})>t/n\}\cap
  \{S_{\ij{1}}>\cdots > S_{\ij{k-1}}\}\right)
\end{align*}
where $M_{\allij} = \max_{i<j; (i,j)\notin\{(i_1,j_1),\cdots,(i_k,j_k)\}}
|S_{ij}|$.  Note that the function $x(x-M_{\allij})$ is increasing for
$x>M_{\allij}$, so the above expression can be rewritten as 
\begin{align*}
  \Pr(T_k>t) &= \sum_{\Ic_k} \Pr\left(|S_{\ij{1}}|(S_{\ij{1}}-M_{\allij}) > \cdots >
  |S_{\ij{k}}|(|S_{\ij{k}}|-M_{\allij}) > t/n\right)\\
  &= \frac{1}{k!}\sum_{\Ic_k} \Pr\left(\bigcap_{\ell \le k}
  \left\{|S_{\ij{\ell}}|(|S_{\ij{\ell}}|-M_{\allij})>t/n\right\}\right),
\end{align*}
where the $k!$ is included because the last event does not distinguish between
permutations of the labels.

As in the previous proof, we now define several events:
\begin{align*}
  E_d &= \left\{\text{The largest $d+1$ elements of $S$ share no indices}\right\}\\
  A_{\allij} &= \bigcap_{\ell \le k}\left\{S_{\ij{\ell}} > M_{\allij}\right\}\\
  A_{\allij}^* &= \bigcap_{\ell \le k}\left\{S_{\ij{\ell}} > M_{\allij}^*\right\}
\end{align*}
where $M_{\allij}^* = \max_{i<j\notin \bigcup_{\ell \le k}
\{i_{\ell},j_{\ell}\}} |S_{ij}|$, the maximum outside of any indices in
$\allij$.  

Then we can expand the above expression in terms of intersections with $E_d$
and $E_d^c$.  Note that
\begin{align*}
  \frac{1}{k!}\sum_{\Ic_k}\Pr\left(\bigcap_{\ell \le k}
  \left\{|S_{\ij{\ell}}|(|S_{\ij{\ell}}|-M_{\allij})>t/n\right\}\middle|E_d^c\right)
  &\le \frac{1}{k!}\sum_{\Ic_k}\Pr\left(A_{\allij}\mid E_d^c\right) = \Pr(E_d^c)
\end{align*}
since the $A_{\allij}$ are disjoint (up to permutations of the indices, which
is accounted for by the $\frac{1}{k!}$).  Similarly,
\begin{align*}
  \sum_{\Ic_k^*}\frac{1}{k!}\Pr\left(\bigcap_{\ell \le k}
  \left\{|S_{\ij{\ell}}|(|S_{\ij{\ell}}|-M^*_{\allij})>t/n\right\}\middle|E_d^c\right)
  &\le \frac{1}{k!}\sum_{\Ic_k^*} \Pr\left(A_{\allij}^*\cap E_d^c\right)\\
  \qquad \qquad \le \Pr(E_d^c) + \frac{1}{k!}\sum_{\Ic_k^*}\Pr\left(A_{\allij}^*\backslash
  A_{\allij}\right).
\end{align*}
Combining these, we obtain
\begin{align*}
  &\left|P(T_k > t) - \frac{1}{k!}\sum_{\Ic_k^*} \Pr\left(\bigcap_{\ell \le k}
  \left\{|S_{\ij{\ell}}|(|S_{\ij{\ell}}|-M^*_{\allij})>t/n\right\}\right)\right|\\
  &\qquad\le 2\Pr(E_d^c) + \sum_{\Ic_k^*}\Pr\left(A_{\allij}^*\backslash
  A_{\allij}\right)/k!.
\end{align*}
where we take advantage of the fact that on $E_d$, $M_{\allij}=M_{\allij}^*$,
and all the terms in the sum over $\Ic_k$ that are not in $\Ic_k^*$ vanish.

Then, by Lemma \ref{lemma:gaplimit} we can replace $\Pr\left(\bigcap_{\ell \le k}
  \left\{|S_{\ij{\ell}}|(|S_{\ij{\ell}}|-M^*_{\allij})>t/n\right\}\right)$,
  obtaining
\begin{align*}
  &\left|P(T_k > t) -  e^{-kt}\left(1+O\left(\frac{1}{\sqrt{\log p}}+\frac{\log
  p}{n}\right)\right) \frac{1}{k!}\sum_{\Ic_k^*} \Pr(A_{\allij}^*)\right|\\
  &\qquad\le 2\Pr(E_d^c) + \frac{1}{k!}\sum_{\Ic_k^*}\Pr\left(A_{\allij}^*\backslash
  A_{\allij}\right) + O\left(e^{-(p-2k)^{3/5}/4\sqrt{\log (p-2k)}}\right)
\end{align*}
By Lemma \ref{lemma:maximadisconnect}, $\Pr(E_d^c) \to 0$, and by Lemma
\ref{lemma:astar}, $\frac{1}{k!}\sum_{\Ic_k^*} \Pr(A_{\allij}^*) \to 1$.  Since
$\frac{1}{k!}\sum_{\Ic_k^*} \Pr(A_{\allij}) \to 1$ and $A_{\allij}\subset
A_{\allij}$,\\
\mbox{$\frac{1}{k!}\sum_{\Ic_k^*} \Pr(A_{\allij}^*\backslash P_{\allij})
\to 0$}.  Combining these, we obtain $P(T_k>t) \to e^{-kt}$, so $T_k\sim
\Exp(1/k)$.  
\end{proof}

The result of the previous theorems is that for a finite number of steps, the
corresponding test statistics have asymptotic distributions $T_k \dlim
\Exp(1/k)$ under the global null that no signal is present.  While this is a
comforting fact, some dependence structure is expected in many of the
situations where the graphical lasso is applied.  In the next section, we relax
the global null hypothesis to a weaker one where some correlation structure is
allowed.

\subsection{Null after signal selection}\label{section:details:aftersignal}
Instead of assuming no dependence structure as in the global null, suppose that
the true structure is restricted to a small subset of the nodes, $\Ac$.  By
this we mean that there exists a fixed-size subset, $\Ac\subset\{1,\dots,p\}$, such that the only
off-diagonal nonzeros in $\Sigma^{-1}$ occur in the $\Ac\times\Ac$ block.  

Then the test statistic $T_k = n\tilde{\rho}_k(\tilde{\rho}_k-\tilde{\rho}_{k+1})$ corresponds
to a test of a weaker null hypothesis.  If $\Vc_k$ is the set of variables
involved in the estimate by step $k$, then the null hypothesis at step $k$ is
that the signal variables have already appeared: $H_k: \Ac \subseteq \Vc_{k-1}$.

For this theorem, we have to be careful here to make sure that the signal on $\Ac$ is strong
enough that those variables are selected first.  Then the first step outside of
$\Ac\times \Ac$ is the same as the first step after the variables in $\Ac$ have
been selected.  The theorem presented here assumes
that the signal variables are strong enough to be selected first; it is followed by two simple conditions under which this
occurs.

We show in the following theorem that under the above condition, and defining $m$ to
be the last step inside $\Ac\times\Ac$, the first null test statistic $T_{m+1} =
\sqrt{n}\tilde{\rho}_{m+1}(\tilde{\rho}_{m+1}-\tilde{\rho}_{m+2}) \dlim \Exp(1)$ and
subsequent steps  $T_{m+k} = \sqrt{n}\tilde{\rho}_{m+k}(\tilde{\rho}_{m+k}-\tilde{\rho}_{m+k+1}) \dlim \Exp(1/k)$, for $1 < k \le d$ and $d$ finite.  

As this theorem relies on Theorems \ref{thm:globalfirststep} and \ref{thm:globallatersteps}, the first
step ($k=1$) is proven without Conjecture \ref{conjecture}, and the later steps
($k>1$) rely on the validity of Conjecture
\ref{conjecture}.  Again, we work with the order statistics $V_1,V_2,\dots$ of the correlation
matrix, as $V_{m+1},\dots,V_{m+d+1}$ are identical to
$\tilde{\rho}_{m+1},\dots,\tilde{\rho}_{m+d+1}$ on the event that the largest
$d+1$ elements of $S$ outside of $\Ac\times\Ac$ share no indices, which has probability limiting to 1.

\begin{thm}\label{thm:latersteps}
  Let $Z_1,\dots,Z_p$ be independent, and define $\Ac\subset \{1,\dots,p\}$ so
  that for $j\notin \Ac$, $Z_j$ are i.i.d. with spherical distribution around
  0.  Let $V_1\ge V_2 \ge \cdots \ge V_{p(p-1)/2}$ be the order statistics of
  the absolute values of the correlations of $Z_1,\dots,Z_p$.

  Let $B$ be the event that there exists a $\delta$ such that for all $i\in \Ac$,
  there exists $j\in\Ac$ such that $S_{ij} > \delta$, and furthermore $S_{ij} <
  \delta$ for all $i\notin \Ac$.  Let $m$ denote the number of steps taken
  before the first edge enters outside of $\Ac\times\Ac$.  If $\Pr(B)\to 1$, then as $n,p\to\infty$ with
  $\frac{\log p}{n}$, $nV_{m+k}(V_{m+k}-V_{m+k+1})\dlim \Exp(1/k)$,
  for $1\le k \le d$ and fixed $d\ge 0$.
\end{thm}
Note that sufficient conditions for $\Pr(B)\to 1$ are either the conditions for
recovery in \cite{ravikumar2008model}, or that each variable in $\Ac$ has at least one
true correlation with limit strictly greater than $\sqrt{\frac{4\log p}{n}}$.

\begin{proof}
  For simplicity of notation, rearrange the variables so that the variables in
  $\Ac$ are the first $|\Ac|$ variables.  Let $E_{\Ac,d}$ be the event that the first
  $d+1$ largest off-diagonal elements of $S$ outside of $\Ac\times \Ac$ share no
  indices, and let $D$ be the event that these $d$ elements have no indices in
  $\Ac$.  By Lemma \ref{lemma:maximadisconnect}, $\Pr(E_{\Ac,d})\to 1$, and by Lemma
  \ref{lemma:signalcasedisconnect}, $\Pr(D)\to 1$.  Therefore $\Pr(B\cap D\cap
  E_{\Ac,d}) \to 1$.

  On $B\cap D\cap E_{\Ac,d}$, $V_{m+k}(V_{m+k}-V_{m+k+1}) =
  \tilde{V}_k(\tilde{V}_k-\tilde{V}_{k+1})$, where $\tilde{V}_1\ge \cdots \ge
  \tilde{V}_{(p-|\Ac|)(p-|\Ac|-1)/2}$ are the order statistics of the
  correlations of variables not in $\Ac$.  Therefore, in general, 
  $V_{m+k}(V_{m+k}-V_{m+k+1}) =
  \tilde{V}_k(\tilde{V}_k-\tilde{V}_{k+1}) + o_P(1)$.  Note that Theorems
  \ref{thm:globalfirststep} and \ref{thm:globallatersteps} apply to
  $\tilde{V}_1,\dots,\tilde{V}_k$, so the rest of the proof follows.
\end{proof}

Note that for this theorem, we just consider the first steps outside of
$\Ac\times\Ac$.  If there are several disconnected components in $\Ac$, it is
more interesting to consider the first step outside of any connected components
in $\Ac\times\Ac$.  However, since $\Ac$ is finite and the number of steps $d$ considered below is
finite, the probability that this case differs from the one addressed goes to
zero (similar to Lemma \ref{lemma:signalcasedisconnect}), so the asymptotic distributions are the same.

\section{Simulations}\label{section:simulations}
We explore the null distribution of the covariance statistic through
simulations.  These are carried out under the global null, where there is no
correlation between the variables, and the weaker null, where correlation
is present only on a subset $\Ac$ of the variables.

\subsection{Simulations under the global null}
We generate vectors $X_1,\dots,X_n \in \R^p$ as $X_i\iid N(0,I_p)$.
Here $n=500$, $p=100$.  The correlation matrix then has entries
$S_{ij} = X_i^TX_j/||X_i||_2||X_j||_2$, after first centering the $X_i$.  The
simulations are all repeated 1000 times.

\begin{figure}[h]
  \centering
  \includegraphics[width=5in]{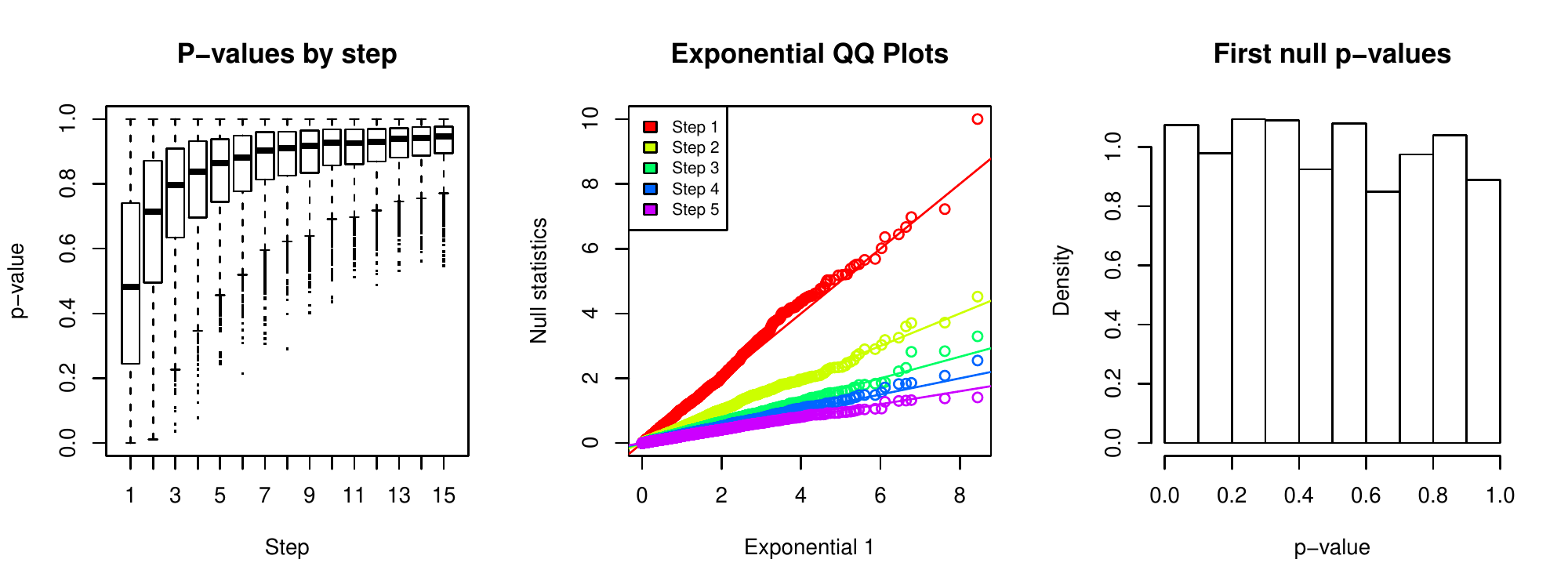}
  \caption{Distribution of test statistics in the global null setting, with
  $n=500, p=100$.  The left panel shows boxplots of the $p$-values, where each
  statistic is compared to the CDF of an $\Exp(1)$ distribution.  Note that
  this is conservative after the first null step. The middle panel shows
  $\Exp(1)$ QQ-plots for the first five steps, demonstrating that the
  distribution is very nearly exponential and has mean $1/k$ for the $k^{th}$
  step.  The right plot shows a histogram of the $p$-values at the first
  step, demonstrating them to be reasonably uniform.}
  \label{fig:globalnullsim}
\end{figure}

Figure \ref{fig:globalnullsim} considers the corresponding covariance statistics.
Only $k=1,\dots,5$ are shown, for simplicity.  The first plot shows box plots of the $p$-values $1-F_1(T_k)$, where each $T_k$ is compared to
the CDF for an $\Exp(1)$ distribution.  Note that this is conservative for all
$k>1$, leading the later $p$-values to quickly approach 1.
The second plot compares the
distributions of $T_1,\dots,T_5$ to an $\Exp(1)$ distribution in a QQ plot.  On
this plot, a line with slope $1/k$ corresponds to an exponential distribution
with mean $1/k$. The last plot shows a histogram of the
$p$-value for the first step, showing it to be close to uniform.

\begin{figure}[h]
  \centering
  \includegraphics[width=3.5in]{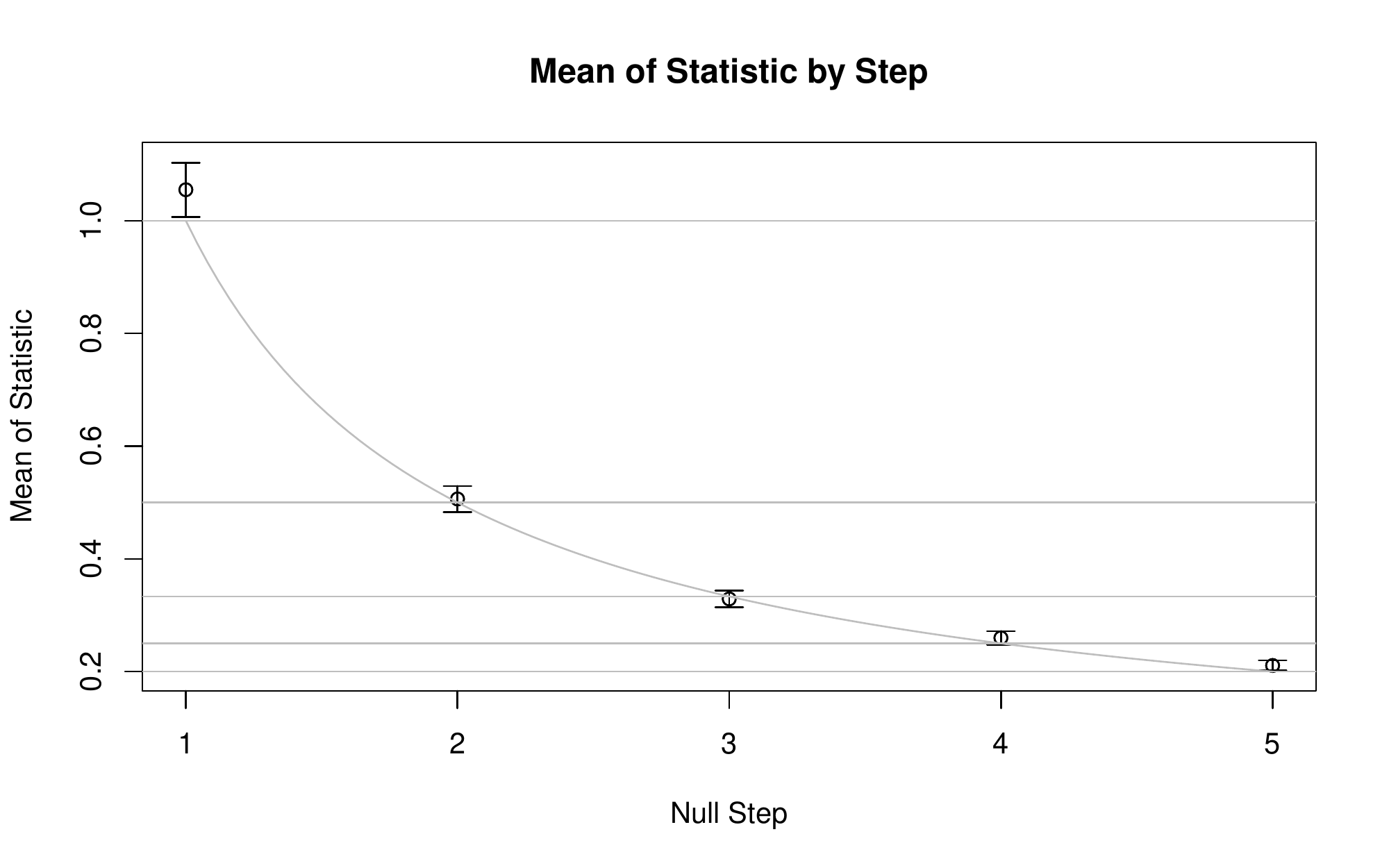}
  \caption{Confidence intervals for the means of the statistics at the first
  five null steps.  These agree with the theoretical prediction that
  the $k^{th}$ step should have mean $1/k$.}
  \label{fig:statmeans}
\end{figure}

To more easily see that the statistic means at step $k$ are matching the
theoretical value of $1/k$, Figure \ref{fig:statmeans} shows 95\% confidence intervals
for the means of the test statistics for the first five steps, based on 1000
realizations.  We see that the empirical means are in agreement with the
theoretical predictions.

\subsection{Simulations with signal, $H_{\Ac}$}
In this simulation, we select a subset $\Ac\subset\{1,\dots,p\}$ of size
$|\Ac|=6$.  The covariance matrix $\Sigma$ is then designed so that the
$\Ac\times\Ac$ block has some correlation structure, while the rest of the
covariance matrix remains diagonal.  We consider two different structures on
$\Ac$:
\begin{enumerate}
  \item \emph{Disconnected Pairs:} The variables
    in $|\Ac|$ are paired, and $\Sigma^{-1}$ is large on those pairs.
  \item \emph{Clique:} All the entries in the $\Ac\times\Ac$ block of
    $\Sigma^{-1}$ are made large.
\end{enumerate}

\begin{figure}[h!]
  \centering
  \includegraphics[width=5in]{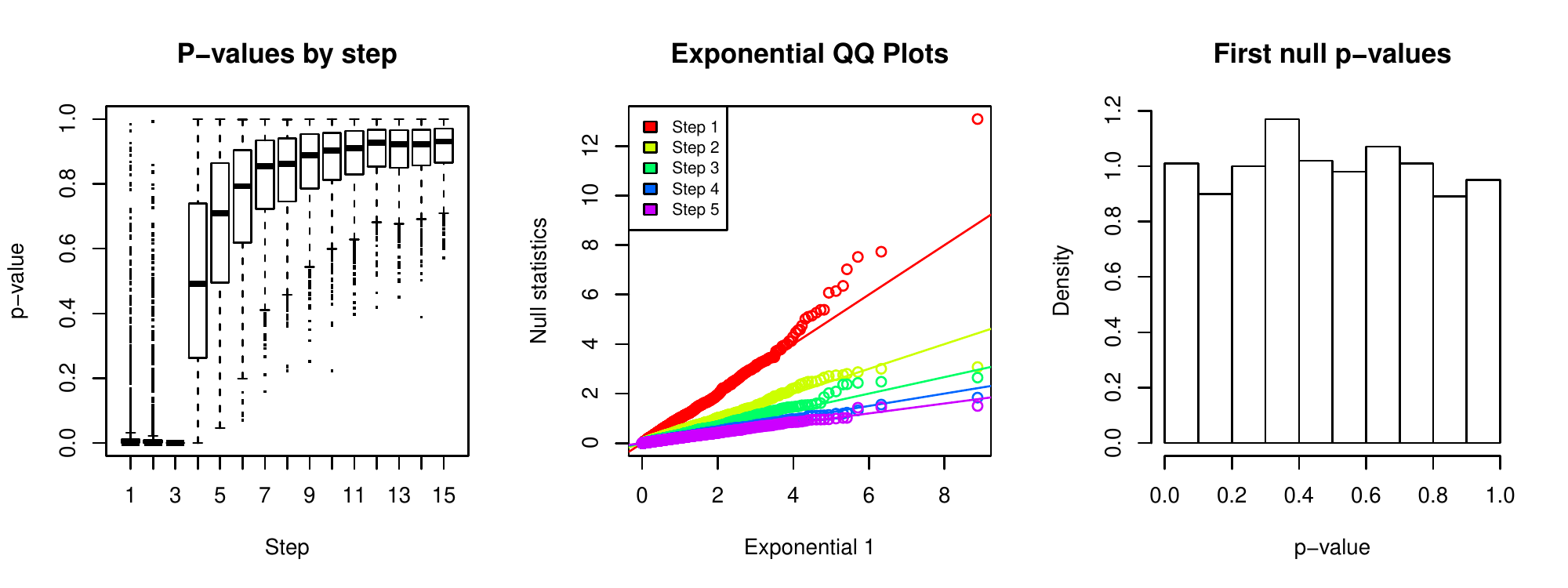}
  \caption{Distribution of test statistics under the alternative of three
  highly correlated pairs of variables, with $n=500$ and $p=100$.  The left
  panel shows boxplots of the $p$-values, where each statistic is compared to
  the CDF of an $\Exp(1)$ distribution.  Note that this is conservative after
  the first null step. The middle panel shows $\Exp(1)$ QQ-plots for the first
  five null steps, demonstrating that the distribution is very nearly
  exponential and has mean $1/k$ for the $k^{th}$ null step.  The right plot shows a histogram of the
  $p$-values at the first
  null step, demonstrating them to be reasonably uniform.}
  \label{fig:pairsim}
\end{figure}

\begin{figure}[h!]
  \centering
  \includegraphics[width=5in]{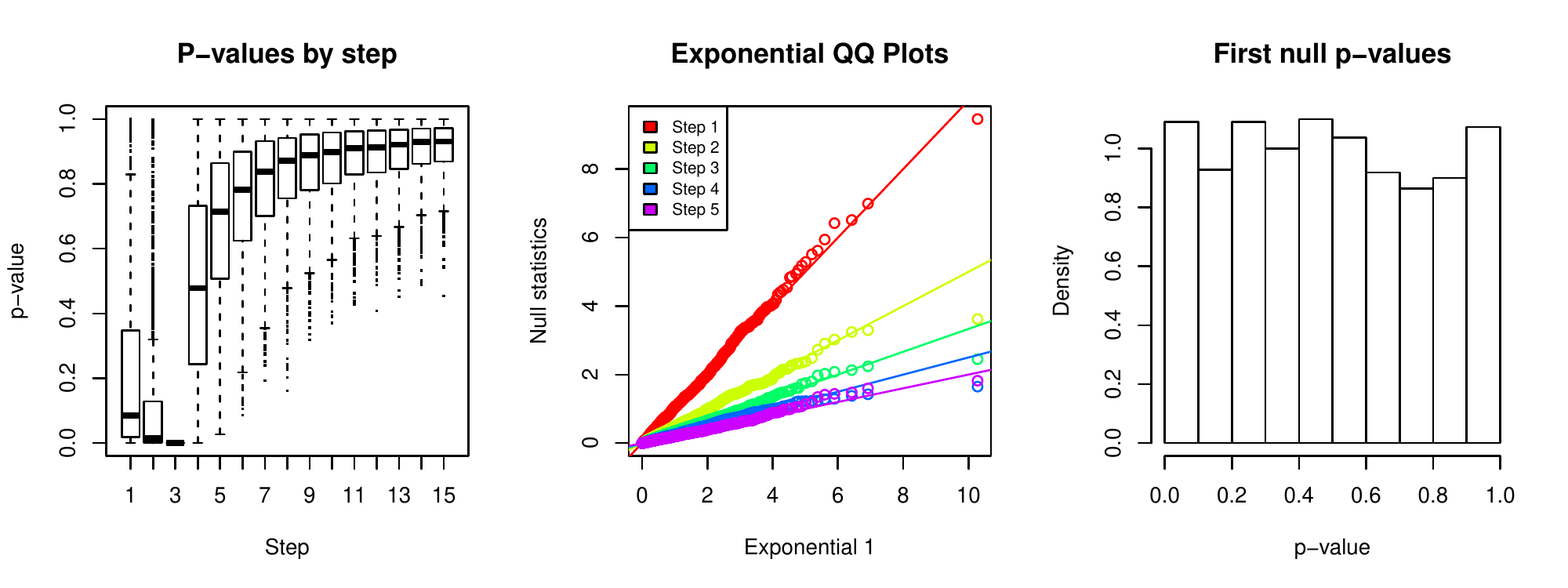}
  \caption{Distribution of test statistics under the alternative of a clique of
  6 highly correlated variables, with $n=500$ and $p=100$.  The left panel
  shows boxplots of the $p$-values, where each statistic is compared to the CDF
  of an $\Exp(1)$ distribution.  Note that this is conservative after the first
  null step. The middle panel shows $\Exp(1)$ QQ-plots for the first five null
  steps, demonstrating that the distribution is very nearly exponential and has
  mean $1/k$ for the $k^{th}$ null step.  The right plot shows a histogram of
  the $p$-values at the first null step, demonstrating them to be reasonably
  uniform.}
  \label{fig:cliquesim}
\end{figure}

The purpose of this section is to investigate whether the statistics retain the desired
null behavior as edges are selected outside of the signal.  Figures \ref{fig:pairsim} and \ref{fig:cliquesim} demonstrate the behavior of our
test statistic under scenarios 1 and 2, respectively.  These simulations are
both run for the case $n=500$ and $p=100$.  The panels of each plot describe
the same quantities as in Figure \ref{fig:globalnullsim}.  

Note that the statistics exhibit the expected behavior: The statistic for the
$k^{th}$ null step appears to closely follow an exponential distribution with
mean $1/k$, as suggested by the theory
in Section \ref{section:details:aftersignal}.

\section{Connection to clustering}

The theoretical results from this paper apply directly to the particular
sequence $\tilde{\rho}_1,\dots,\tilde{\rho}_M$ of absolute correlations
described in Section \ref{glassotest:path}.  This allows the results from this
paper to be applied to other statistical methods which rely on the same sequence.  Hierarchical single-linkage clustering based
on correlations is one of these methods.

\begin{figure}[h!]
  \centering
  \includegraphics[width=3.5in]{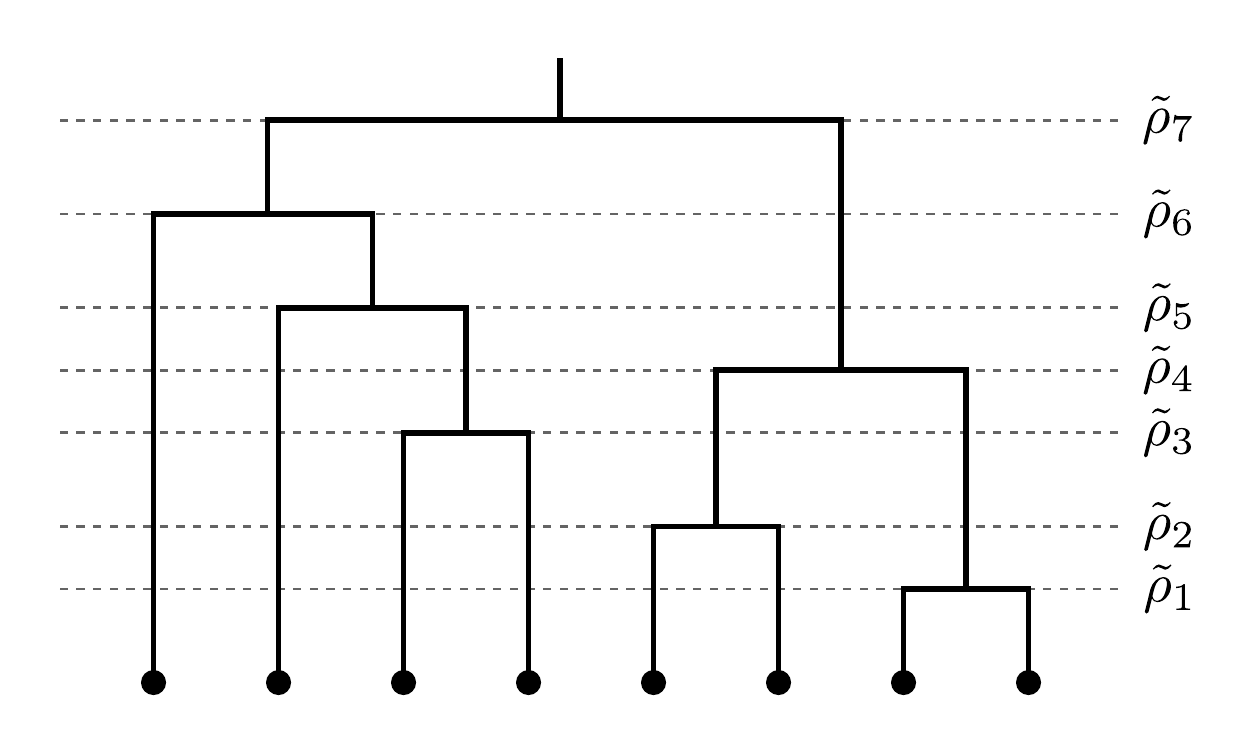}
  \caption{Illustration of single linkage clustering of correlations and the
  correspondence of levels in the tree to $\tilde{\rho}_1,\tilde{\rho}_2,\dots$
  in the graphical lasso setting.  Because of this correspondence, the
  theoretical results from this paper also apply to single linkage clustering
  based on correlations.}
  \label{fig:hiertree}
\end{figure}

Consider single linkage clustering of variables, based on their absolute
pairwise correlations.  Starting with all the variables disconnected, variable
groups are merged at the level of the tree equal to the largest correlation
between the two groups.  The subtrees that arise in such a fashion correspond
exactly to the connected components in the graphical lasso path.  The sequence
of levels at which the merges occur is exactly the sequence
$\tilde{\rho}_1,\dots,\tilde{\rho}_M$, illustrated in Figure
\ref{fig:hiertree}.  As a result, all the theoretical results from this paper
apply to this clustering setting as well.

This connection implies that the test statistic $T_k$ corresponds to the $k^{th}$ merge in
the tree generation process, and it is a function of the height of the tree at
that merge and at the following merge.  The null hypotheses being tested by
$T_k$ is the hypothesis that there are no correlated variables in separate
subtrees at this level of the tree or above.  The statistic $T_k$ will have the
same asymptotic exponential null distribution that we proved in section
\ref{section:details}.

\section{Discussion}\label{section:discussion}

The graphical lasso provides an attractive way to estimate sparse inverse
covariance matrices.  However, approaches to integrating hypotheses testing and
inference into this setting have not yet been developed.  Taking inspiration
from recent results from \cite{lockhart2013significance} on inference in lasso
estimation, this paper constructs a series of hypotheses and test statistics along the
graphical lasso path.

For these test statistics, simple asymptotic null distributions are proven under the
corresponding sequence of hypotheses.  It is also demonstrated empirically that
these asymptotic distributions provide a good approximation to the finite
sample distributions of these statistics, in several signal settings.  Finally,
extensions of these results are made to other correlation-based methods; in
particular, single-linkage clustering based on absolute correlations.

The approach presented in this paper tests the hypotheses that all the variables that should
be connected have been connected, in the sense that they are in the same
connected component.  Because of the subset of graphical lasso knots that are
used, no statements can be made about the internal structure of the estimate
on each connected component.

Similarly, we have only made statements about the distribution of this
statistic under the null distribution.  For practical application of this test,
it is important to also understand how it behaves under the alternative.  Based
on our empirical demonstrations, the statistic tends to be large in the
presence of strong signal, leading to very small $p$-values when compared to
the $\Exp(1)$ distribution.  This makes sense intuitively, as
$n\tilde{\rho}_k(\tilde{\rho}_k-\tilde{\rho}_k)$ should grow large as
$\tilde{\rho}_k, \tilde{\rho}_{k+1}$ limit to finite, nonzero true correlations.

\begin{figure}[h!]
  \centering
  \includegraphics[width=3in]{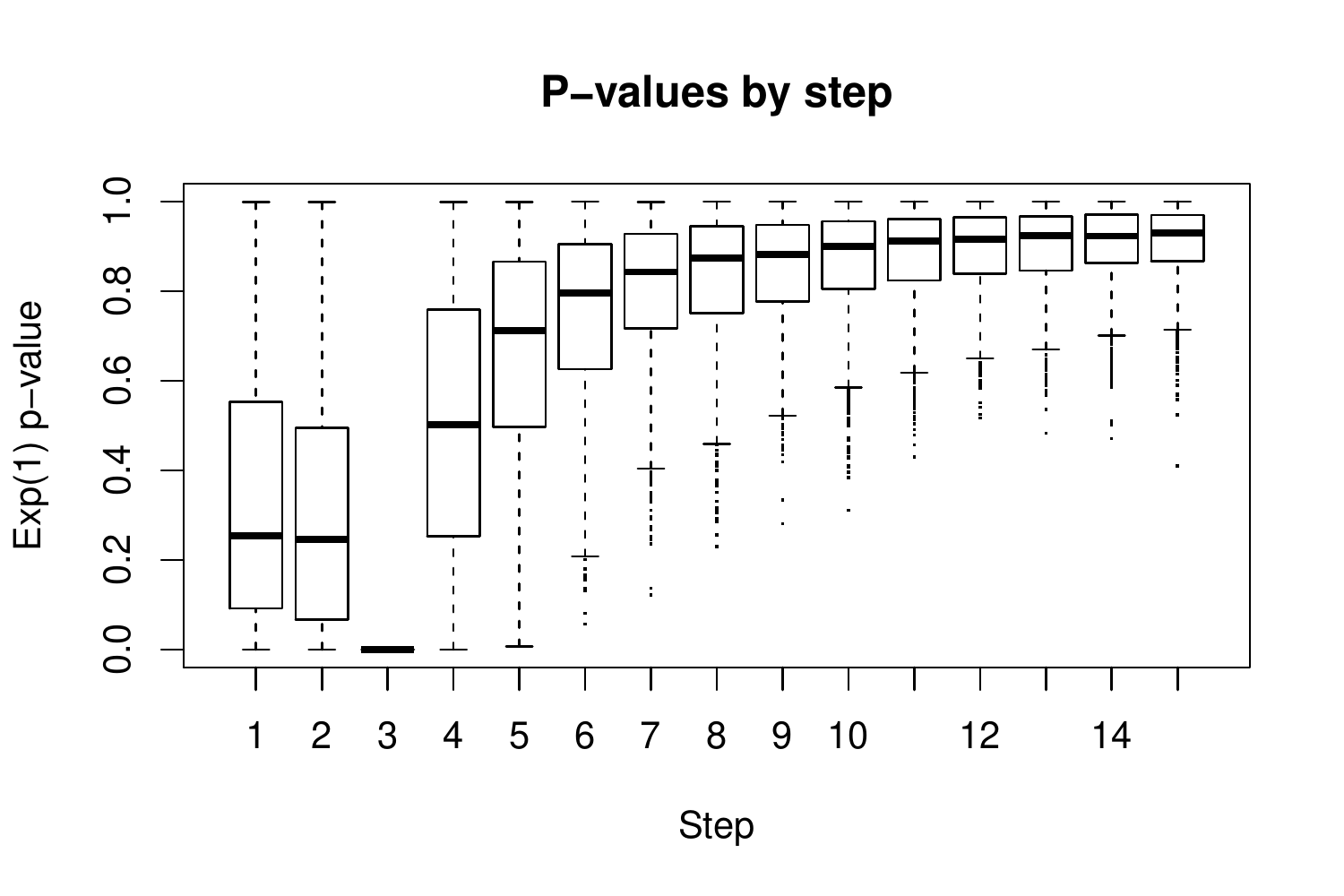}
  \caption{Simulation results when there are three highly correlated pairs of
  variables with exactly the same true nonzero correlation.  Boxplots of the $p$-values
  are shown, where each statistic is compared to the CDF
  of an $\Exp(1)$ distribution.  Note that the third and later steps behave as
  expected, but the $p$-values at the first two steps are larger than in Figure
  \ref{fig:pairsim}.}
  \label{fig:breaksim}
\end{figure}

An special case of this occurs if two strong correlations have the same true
nonzero value.  Then the difference $\tilde{\rho}_k-\tilde{\rho}_k$ becomes
smaller as $n$ increases, slowing the growth of the statistic.  The result is
that the statistic grows at rate $\sqrt{n}$, rather than the usual rate $n$.
In practice, this can be seen with finite sample sizes and close nonzero
correlations.  This is illustrated in Figure \ref{fig:breaksim}, where the
simulation from Figure \ref{fig:pairsim} is repeated with an alternative of
three pairs with exactly the same strong correlation. There
the third step and all the null steps behave as before.  However, the
statistics for the first two non-null steps are quite large.  This same
behavior appears in the equivalent regression test with the lasso, and research
is underway in determining good ways to handle this.

Finally, our eventual hope is to use results like those presented here and in
\cite{lockhart2013significance} to construct stopping rules for regularized
regression methods like the lasso and the graphical lasso.  The very strong
similarities between the forms of the test statistics and null distributions
between these two settings suggest that an overarching approach can be
constructed that will apply broadly across these methods.

\section*{Acknowledgements}
The authors would like to thank Noah Simon for pointing out the connection
between the graphical lasso and hierarchical clustering, and Jacob Bien for
helpful feedback.

\appendix

\section{Supporting Lemmas}

This appendix contains the supporting lemmas for Theorems 1,2 and 3 from the
paper, along with their proofs.

Throughout this section, we use the following notation:
\begin{itemize}
  \item $f_n(x) = c_n\left(1-\frac{x^2}{n}\right)^{(n-4)/2}$, supported on
    $[0,\sqrt{n}]$, where $c_n =
    \frac{2}{\sqrt{\pi}}\frac{\Gamma\left(\frac{n-1}{2}\right)}{\Gamma\left(\frac{n-2}{2}\right)}
    = \sqrt{\frac{2}{\pi}}\left(1+O\left(\frac{1}{n}\right)\right)$.  This is
    the marginal density of $\sqrt{n}|S_{ij}|$.
  \item $F_n(x) = \int_0^x f_n(u)du$ and $\bar{F}_n(x) = \int_x^{\sqrt{n}}
    f_n(u)du$.
  \item $G_{n,p}(x) = \Pr(\sqrt{n}M_{n,p} < x)$, where $M_{n,p}$ is the maximum of a
    $p\times p$ correlation matrix of spherically distributed variables (as in
    Theorem 1).  $g_{n,p}(x)$ is the corresponding density.
  \item $f(x) = O(g(x))$ if and only if $\exists M\in \R^+$, $x_0\in\R$ such that $|f(x)| \le
M|g(x)|$ for all $x>x_0$.
\end{itemize}

\subsection{Bounds on marginal distributions}

This section presents results bounding the tails of the marginal distributions
for $\sqrt{n}|S_{ij}|$ and $\sqrt{n}M_{n,p}$.  Lemma \ref{lemma:mills} gives a
Mills ratio bound for the tail of $\sqrt{n}|S_{ij}|$, and Lemma
\ref{lemma:exp_limit} gives an important consequence of it.  Lemmas
\ref{lemma:Mexplowerbound} and \ref{lemma:chebyshev} give bounds on the
distributions of large elements of the correlation matrix.  Lemma
\ref{lemma:gaplimit} combines these elements to provide an important result for
proving the theorems in the paper.

\begin{lemma}
  \label{lemma:mills}
  Let $f_n$ be the density $f_n(x) =
\frac{2}{\sqrt{n\pi}}\frac{\Gamma(\frac{n-1}{2})}{\Gamma(\frac{n-2}{2})}\left(1-\frac{x^2}{n}\right)^{(n-4)/2}$
supported on $\left[0,\sqrt{n}\right]$ and let $\bar{F}_n(x) =
\int_x^{\sqrt{n}}f_n(w)dw$.
Then for $n\ge 3$, 
\begin{align}
  \frac{\bar{F}_n(x)}{f_n(x)} &\le \frac{n}{n-2} \cdot \frac{1}{x}
  \left(1-\frac{x^2}{n}\right) \label{upperbound} & \text{for }x\in (0,\sqrt{n})\\
  \frac{\bar{F}_n(x)}{f_n(x)} &\ge \frac{n+1}{n-2}\cdot \frac{x}{x^2+1}
  \left(1-\frac{x^2}{n}\right) & \text{for }x \in (a_n,\sqrt{n}) \label{lowerbound}
\end{align}
where $a_n$ are constants satisfying $a_n\le \sqrt{3/5}$.
\end{lemma}

\begin{proof}
  This proof follows the approach of \cite{millsproof}.  Define $r_n(x) =
  \bar{F}_n(x)/f_n(x)$.

  To prove (\ref{upperbound}), define $h_1(x) =
  \frac{1}{x}\left(1-\frac{x^2}{n}\right)$.  Note that $h_1(x)f_n(x) \propto
 \frac{1}{x}\left(1-\frac{x^2}{n}\right)^{(n-2)/2}$ is strictly decreasing and
nonzero on $(0,\sqrt{n})$.  Also note that
$\lim_{x\to\sqrt{n}} \bar{F}_n(x) = 0$ and $\lim_{x\to\sqrt{n}} h_1(x)f_n(x) = 0$.

Thus, by L'Hopital's rule,
\begin{align*}
  \lim_{x\to\sqrt{n}} \frac{r_n(x)}{h_1(x)} =
  \lim_{x\to\sqrt{n}}\frac{\bar{F}_n(x)}{h_1(x)f_n(x)} = \lim_{x\to\sqrt{n}}
  \frac{\bar{F}'_n(x)}{[h_1(x)f_n(x)]'}.
\end{align*}

Defining $g_1(x) \equiv \frac{\bar{F}'_n(x)}{[h_1(x)f_n(x)]'}$, we have
\begin{align*}
  g_1(x) = \frac{-f_n(x)}{f'_n(x)h_1(x) + f_n(x)h_1'(x)} =
  \frac{1}{\frac{n-3}{n}+\frac{1}{x^2}},
\end{align*}
and $\lim_{x\to\sqrt{n}}g_1(x) = \frac{n}{n-2}$.

Therefore, $\lim_{x\to\sqrt{n}} r_n(x)/h_1(x) = \frac{n}{n-2}$.
Furthermore, by Lemma 2.1 of \cite{millsproof}, since $g_1(x)$ is strictly
increasing on $(0,\sqrt{n})$, $r_n(x)/h_1(x)$ is also strictly increasing on that
interval.  Combining these facts gives $r_n(x)/h_1(x) \le \frac{n}{n-2}$ or
equivalently
\begin{align*}
  r_n(x) = \frac{\bar{F}_n(x)}{f_n(x)} \le \frac{n}{n-2}\cdot
  \frac{1}{x}\left(1-\frac{x^2}{n}\right)
\end{align*}
for $x\in(0,\sqrt{n})$.

To prove (\ref{lowerbound}), we similarly define $h_2(x) =
\frac{x}{1+x^2}\left(1-\frac{x^2}{n}\right)$.  Now $h_2(x)f_n(x)$ nonzero with
constant derivative on an interval $(a_n,\sqrt{n})$ for $n\ge 3$, with $a_3 =
\sqrt{3/5}$ and $a_n = \sqrt{\frac{\sqrt{8n^2-16n+1}-2n+1}{2(n-3)}}$ for $n\ge
4$.  In particular, $a_n \le
\sqrt{3/5}$ for all $n\ge 3$, $a_n$ is decreasing in $n$, and
$\lim_{n\to\infty} a_n = \sqrt{\sqrt{2}-1}$.  Furthermore, $\lim_{x\to\sqrt{n}}
h_2(x)f_n(x) = 0$.

We now have
\begin{align*}
  g_2(x) = \frac{\bar{F}'_n(x)}{[f_n(x)h_2(x)]'} =
  \frac{(1+x^2)^2}{\frac{n-3}{n}\cdot x^4 + \frac{2n-1}{n}\cdot x^2 - 1},
\end{align*}
where $g_2(x)$ is strictly decreasing on $(a_n,\sqrt{n})$ and
$\lim_{x\to\sqrt{n}}g_2(x) = \frac{n+1}{n-2}$.

Then, by L'Hopitals rule, $\lim_{x\to\sqrt{n}} r_n(x)/h_2(x) =
\lim_{x\to\sqrt{n}} g_2(x) = \frac{n+1}{n-2}$, and furthermore by Lemma 2.1 of
\cite{millsproof}, $r_n(x)/h_2(x)$ is strictly decreasing on $(a_n,\sqrt{n})$.
Therefore $r_n(x)/h_2(x) \ge \frac{n+1}{n-2}$ on that interval, so
\begin{align*}
  r_n(x) = \frac{\bar{F}_n(x)}{f_n(x)} \ge
  \frac{n+1}{n-2}\cdot\frac{x}{x^2+1}\left(1-\frac{x^2}{n}\right)
\end{align*}
for $x \in (a_n,\sqrt{n})$.
\end{proof}

\begin{lemma}
  \label{lemma:exp_limit}
  For $\sqrt{n}|S_{ij}| \sim f_n(x)$, $t>0$, and $x\in (\sqrt{\log p},\sqrt{n})$,
  as $n,p\to\infty$ with $\frac{\log p}{n} \to 0$ we have
  \begin{align*}
    \frac{\Pr\left(n^{1/2}|S_{ij}|\left(n^{1/2}|S_{ij}|-x\right) \ge t\right)}{\Pr\left(n^{1/2}|S_{ij}|\ge
    x\right)} &= e^{-t}\left(1+O\left(\frac{1}{\sqrt{\log p}} + \frac{\log
    p}{n}\right)\right)
  \end{align*}
  and
  \begin{align*}
    \frac{\Pr\left(n^{1/2}|S_{ij}|\ge x +
    \frac{t}{x}\right)}{\Pr\left(n^{1/2}|S_{ij}|\ge
    x\right)} &= e^{-t}\left(1+O\left(\frac{1}{\sqrt{\log p}} + \frac{\log
    p}{n}\right)\right).
  \end{align*}
\end{lemma}

\begin{proof}
  First, note that $|S_{ij}|(|S_{ij}|-x) > t/n \Leftrightarrow |S_{ij}| > x +
  \frac{x}{2}\left(\sqrt{1-\frac{4t}{x^2}}-1\right)$.  Applying the Lagrange
  form of the Taylor series remainder, this is equivalent to $|S_{ij}| > x +
  \frac{x}{2}\left(\frac{2t}{x^2} -
  \frac{1}{(1+\xi)^{3/2}}\frac{4t^2}{x^4}\right)$ for $\xi \in
  [0,\frac{4t}{x^2}]$.  With some rearrangement, this becomes $|S_{ij}| >
  \frac{t}{x} - \frac{2t^2\gamma}{x^3}$ for a particular $\gamma \in
  [\frac{1}{(1+4t/x^2)^{3/2}},1] \subset [0,1]$.

  Next, from \cite{corrlimit}, we know that $\sqrt{n}|S_{ij}|$ has distribution 
  $f_n(x) =
  \frac{2}{\sqrt{n\pi}}\frac{\Gamma(\frac{n-1}{2})}{\Gamma(\frac{n-2}{2})}\left(1-\frac{x^2}{n}\right)^{(n-4)/2}$
  with support on $(0,\sqrt{n})$, so that Lemma \ref{lemma:mills} applies.  Let
  \begin{align*}
    U(z) = \frac{n}{n-2}\frac{1}{z}\left(1-\frac{z^2}{n}\right),\qquad L(z) =
    \frac{n+1}{n-2}\frac{z}{z^2+1}\left(1-\frac{z^2}{n}\right),
  \end{align*}
  the quantities from Lemma \ref{lemma:mills}, so that
  \begin{align*}
    L(z)f(z) \le \bar{F}(z) = \Pr(\sqrt{n}S_{ij}\ge z) \le U(z)f(z).
  \end{align*}

  Evaluating this expression at $z = x + \frac{t}{x} - \frac{2t^2\gamma}{x^3}$ and $z=x$ and dividing
  yields the bounds
  \begin{align*}
    \frac{L(x + \frac{t}{x} - \frac{2t^2\gamma}{x^3})}{U(x)}\cdot
    \frac{f(x+\frac{t}{x} - \frac{2t^2\gamma}{x^3})}{f_n(x)} &\le \frac{\Pr\left(n^{1/2}|S_{ij}|\ge x +
    \frac{t}{x} - \frac{2t^2\gamma}{x^3}\right)}{\Pr\left(n^{1/2}|S_{ij}|\ge
    x\right)}\\
    &\le \frac{f(x+\frac{t}{x} - \frac{2t^2\gamma}{x^3})}{f_n(x)}\cdot\frac{U(x+\frac{t}{x} - \frac{2t^2\gamma}{x^3})}{L(x)}.
  \end{align*}
  Since $x \ge \sqrt{\log p}$, as $(\log p)/n \to \infty$ these expressions  
  behave like
  \begin{align*}
    \frac{f(x+\frac{t}{x} - \frac{2t^2\gamma}{x^3})}{f_n(x)} &= \left(\frac{1-\frac{1}{n}(x +
    \frac{t}{x} - \frac{2t^2\gamma}{x^3})^2}{1-x^2/n}\right)^{(n-4)/2} = e^{-t}\left(1+O\left(\frac{1}{\log p} + \frac{\log p}{n}\right)\right)\\
    \frac{U(x+\frac{t}{x} - \frac{2t^2\gamma}{x^3})}{L(x)} &=
    \frac{\frac{n}{n-2}\frac{1}{x+\frac{t}{x} -
    \frac{2t^2\gamma}{x^3}}\left(1-\frac{(x+\frac{t}{x} - \frac{2t^2\gamma}{x^3})^2}{n}\right)}{\frac{n+1}{n-2}\frac{x}{(1+x)^2}\left(1-x^2/n\right)}
    = 1 + O\left(\frac{1}{\sqrt{\log p}} + \frac{\log p}{n}\right)\\
    \frac{L(x+\frac{t}{x} - \frac{2t^2\gamma}{x^3})}{U(x)} &= \frac{\frac{n+1}{n-2}\frac{x +
    \frac{t}{x} - \frac{2t^2\gamma}{x^3}}{(x+\frac{t}{x} -
    \frac{2t^2\gamma}{x^3})^2+1}\left(1-\frac{1}{n}(x+\frac{t}{x} - \frac{2t^2\gamma}{x^3})^2\right)}{\frac{n}{n-2}\frac{1}{x}\left(1-\frac{x^2}{n}\right)}
    =  1 + O\left(\frac{1}{\sqrt{\log p}} + \frac{\log p}{n}\right)
  \end{align*}
  and so 
  \begin{align*}
    \frac{\Pr\left(n^{1/2}|S_{ij}|\left(n^{1/2}|S_{ij}|-x\right)>t\right)}{\Pr\left(n^{1/2}|S_{ij}|\ge x\right)} =
    e^{-t}\left(1+O\left(\frac{1}{\sqrt{\log p}} + \frac{\log
    p}{n}\right)\right).
  \end{align*}
  The same argument holds with $\gamma = 0$, yielding 
  \begin{align*}
    \frac{\Pr\left(n^{1/2}|S_{ij}|\ge x +
    \frac{t}{x}\right)}{\Pr\left(n^{1/2}|S_{ij}|\ge x\right)} =
    e^{-t}\left(1+O\left(\frac{1}{\sqrt{\log p}} + \frac{\log
    p}{n}\right)\right).
  \end{align*}
\end{proof}

\begin{lemma}
  \label{lemma:Mexplowerbound}
  Let $S$ be a $p\times p$ correlation matrix, with underlying data
  $Z_1,\dots,Z_p \in \R^n$, where the $Z_j$ are independent and identically
  spherically distributed around the origin.  Then
  \begin{align*}
    \Pr(\max_{i<j} \sqrt{n}|S_{ij}| < \sqrt{\log p}) \le
    \exp\left(-\frac{p^{3/5}}{4\sqrt{\log p}}\right).
  \end{align*}
\end{lemma}
\begin{proof}
  Let $\tilde{S}_j = S_{2j-1,2j}$, $j=1,\dots,\floor{\frac{p}{2}}$.  Note that the $\tilde{S}_j$ are mutually
  independent, and that $\sqrt{n}\tilde{S}_j \sim f_n(x)$.  Then
  \begin{align*}
    \Pr(\max_{i<j}\sqrt{n}|S_{ij}| < \sqrt{\log p}) &\le
    \left(\Pr(|\tilde{S}_j| \le \sqrt{\log
    p})\right)^{\floor{\frac{p}{2}}}
    = \left(1-\bar{F}_n(\sqrt{\log p})\right)^{\floor{\frac{p}{2}}}
  \end{align*}
  Using Lemma \ref{lemma:mills}, 
  \begin{align*}
    \bar{F}_n(\sqrt{\log p}) &\ge \frac{1}{\sqrt{\log
    p}}\left(1-\frac{\log p}{n}\right)^{n/2-1}\left(1+O\left(\frac{1}{\log p}\right)\right)\\
    &\ge \frac{1}{\sqrt{\log p}}\left(1+O\left(\frac{1}{\log p}\right)\right)\exp\left(-\frac{\log p}{2}-\frac{\log^2
    p}{4n(1-\xi)^2}\right)
  \end{align*}
  Here $\xi \in [0,\frac{\log p}{n}]$ from applying the Taylor series remainder
  theorem to the expansion of $\log(1-x)$ around 0.  Suppose that $\log p/n < 1/4$,
  which must happen eventually since $\log p/n \to 0$.  Then $\frac{\log^2
  p}{4n(1-\xi)^2} \le \frac{1}{9}\log p$ and therefore $\exp\left(-\frac{\log^2
  p}{4n(1-\xi)^2}\right) \ge p^{-1/9}$, so 
  \begin{align*}
    \bar{F}_n(\sqrt{\log p}) &\ge \frac{1}{\sqrt{\log
    p}}p^{-13/36}\left(1+O\left(\frac{1}{\log
    p}\right)\right) \ge \frac{1}{2\sqrt{\log p}}p^{-2/5}
  \end{align*}
  where the last inequality assumes $p \ge 8$ and uses $2/5> 13/36$ for a
  simpler expression.  Plugging this back into the initial expression and
  assuming $p$ even for notational convenience, 
  \begin{align*}
    \Pr\left(\max_{i<j}\sqrt{n}S_{ij} < \sqrt{\log p}\right) &\le \left(1-\frac{1}{2\sqrt{\log
    p}}p^{-2/5}\right)^{p/2} \le \exp\left(-\frac{p^{3/5}}{4\sqrt{\log p}}\right)
  \end{align*}
\end{proof}

\begin{lemma}
  \label{lemma:chebyshev}
  Let $S$ be a $p\times p$ correlation matrix, with underlying data
  $Z_1,\dots,Z_p \in \R^n$, where the $Z_j$ are independent and identically
  spherically distributed around the origin.  
  
  Let $\tilde{Z} = \sum_{i<j}
  \left(\bar{F}_n(x) - I_{ij}^{(x)}\right)$, where $I_{ij}^{(x)} =
  \Ind_{\{\sqrt{n}|S_{ij}|>x\}}$.  The third moment of $|\tilde{Z}|$ can be computed for
  use in a third-moment Chebyshev bound, which results in the bound
  \begin{align*}
    \Pr\left(\sum_{i<j} I_{ij}^{(x)} < k\right) \le
    \frac{1}{\binom{p}{2}^2\bar{F}_n^2(x)}\cdot 
    \frac{1+4(p-3)\bar{F}_n(x)}{\left(1-\frac{k}{\binom{p}{2}\bar{F}_n(x)}\right)^3},
  \end{align*}
  which simplifies when $x < \sqrt{3.5\log p}$ to 
  \begin{align*}
    \Pr\left(\sum_{i<j} I_{ij}^{(x)} < k\right) \le
    \frac{1}{\binom{p}{2}^2\bar{F}_n^2(x)}\cdot
    \left(1 + O\left(\frac{\sqrt{\log p}}{p^{1/4}}\right)\right)
  \end{align*}
\end{lemma}
\begin{proof}
  The sums involved in $\E \tilde{Z}^3$ are simplified by the pairwise independence of
  the $I_{ij}^{(x)}$ (due to the pairwise independence of the correlations in
  $S$).  The only terms that are nonzero are those where all the terms
  correspond to the same indices $ij$ (of which there are $\binom{p}{2}$), and those where the terms correspond to
  a cycle involving three indices $ij, jk, ki$ (of which there are
  $p(p-1)(p-2)$).  The expectation of the first
  is easy to compute.
  \begin{align*}
    \E\left|\bar{F}_n(x) - I_{ij}^{(x)}\right|^3 &=
    \bar{F}_n(x)\left|\bar{F}_n(x)-1\right|^3 + (1-\bar{F}_n(x))\bar{F}_n^3(x)\\
    &= \bar{F}_n(x)(1-\bar{F}_n(x))\left(\bar{F}_n^2(x)+(1-\bar{F}_n(x))^2\right)\\
    &\le \bar{F}_n(x)
  \end{align*}
  The second is harder to compute, but can be bounded.  Note first that
  \begin{align*}
    \Pr(I_{jk}+I_{ki} > 0| I_{ij}) &\le 2\Pr(I_{jk}=1|I_{ij}) = 2\bar{F}_n(x)
  \end{align*}
  by applying a union bound and pairwise independence.  Furthermore, if at
  least one of $I_{jk}$ and $I_{ki}$ are nonzero, then 
  \begin{align*}
    \E\left(|\bar{F}_n(x)-I_{jk}^{(x)}|\cdot|\bar{F}_n(x)-I_{ki}^{(x)}|\middle|I_{jk}+I_{ki}>0\right)
    &\le 1-\bar{F}_n(x)
  \end{align*}
  Knowing this, we can condition on $I_{ij}$ and expand the original
  expectation to obtain a bound.
  \begin{align*}
    &\E\left|\left(\bar{F}_n(x)-I_{ij}^{(x)}\right)\left(\bar{F}_n(x)-I_{jk}^{(x)}\right)\left(\bar{F}_n(x)-I_{ki}^{(x)}\right)\right|\\
    &\quad = 2\bar{F}_n(x)(1-\bar{F}_n(x))\left( (1-2\bar{F}_n(x))\bar{F}_n^2(x) +
    2\bar{F}_n(x)(1-\bar{F}_n(x))\right)\\
    &\quad \le 2\bar{F}_n^2(x)
  \end{align*}
  Using all this, we can bound $\E |\tilde{Z}|^3$  
  \begin{align*}
    \E |\tilde{Z}|^3 &\le \binom{p}{2}\bar{F}_n(x) +
    4\binom{p}{2}(p-3)\bar{F}_n^2(x).
  \end{align*}
  Then applying Chebyshev's inequality to $\Pr(\sum_{i<j} I_{ij}^{(x)} < k) =
  \Pr\left(\tilde{Z} > \binom{p}{2}\bar{F}_n(x) - k\right)$, we obtain
  \begin{align*}
    \Pr\left(\sum_{i<j} I_{ij}^{(x)} < k\right) \le \frac{\binom{p}{2}\bar{F}_n(x)
    + 4\binom{p}{2}(p-3)\bar{F}_n^2(x)}{\left(\binom{p}{2}\bar{F}_n(x) - k\right)^3}
    = \frac{1}{\binom{p}{2}^2\bar{F}_n^2(x)}\cdot 
    \frac{1+4(p-3)\bar{F}_n(x)}{\left(1-\frac{k}{\binom{p}{2}\bar{F}_n(x)}\right)^3}
  \end{align*}
\end{proof}

\begin{lemma}
  \label{lemma:gaplimit}
  Consider mutually independent random variables $\tilde{S}_1,\dots,\tilde{S}_k$ with
  distribution $f_n$ and $M_{n,p}$ which is distributed as
  $\max_{i<j}\sqrt{n}|S_{ij}|$ for $S_{ij}$ as in Theorem
  \ref{thm:globalfirststep}.  Then
  \begin{align*}
    &\Pr\left(\bigcap_i \left\{\tilde{S}_i(\tilde{S}_i-M_{n,p})>t\right\}\right)\\
    &\qquad= e^{-kt}\Pr\left(\bigcap_i\left\{\tilde{S}_i>M_{n,p}\right\}\right)\left(1+O\left(\frac{1}{\sqrt{\log
    p}} + \frac{\log p}{n}\right)\right)+ O\left(e^{-\frac{p^{3/5}}{4\sqrt{\log p}}}\right)
  \end{align*}
\end{lemma}
\begin{proof}
  Let $g_{n,p}(x)$ be the density of $M_{n,p}$.  We can expand the probability above by
  conditioning on $M_{n,p}$ and taking advantage of the mutual independence of the
  variables.  We obtain
  \begin{align*}
    \Pr\left(\bigcap_i \left\{\tilde{S}_i(\tilde{S}_i-M_{n,p})>t\right\}\right) &=
    \int_0^{\sqrt{n}} \Pr\left(\bigcap_i
    \left\{\tilde{S}_i(\tilde{S}_i-x)>t\right\}\middle|M_{n,p}=x\right)g_{n,p}(x)dx\\
    &= \int_0^{\sqrt{n}} \Pr\left(\tilde{S}_i(\tilde{S}_i-x)>t\right)^k g_{n,p}(x)dx\\
  \end{align*}
  Splitting the integral at $\sqrt{\log p}$, the first piece can be bounded
  by Lemma \ref{lemma:Mexplowerbound}, and the integrand of the second can be simplified by Lemma
  \ref{lemma:exp_limit}, yielding 
  \begin{align*}
    &\Pr\left(\bigcap_i \left\{\tilde{S}_i(\tilde{S}_i-M_{n,p})>t\right\}\right) \\
    &\qquad = O\left(e^{-\frac{p^{3/5}}{4\sqrt{\log p}}}\right) +
    \int_{\sqrt{\log p}}^{\sqrt{n}}
    e^{-kt}\Pr\left(\tilde{S}_i>x\right)^k\left(1+O\left(\frac{1}{\sqrt{\log p}} + \frac{\log p}{n}\right)\right)
    g_{n,p}(x)dx\\
    &\qquad=  e^{-kt}\Pr\left(\bigcap_i\left\{\tilde{S}_i>M_{n,p}\right\}\right)\left(1+O\left(\frac{1}{\sqrt{\log
    p}} + \frac{\log p}{n}\right)\right)+ O\left(\exp^{-\frac{p^{3/5}}{4\sqrt{\log p}}}\right)
  \end{align*}

\end{proof}

\subsection{Locations of large correlations}

The results in this section address the locations of large
elements within the correlation matrices we consider.  The first shows that
with high probability, the largest correlations will have no overlapping
variables.  The
second lemma shows that, with high probability, the largest correlations will not involve the finite
set of variables, $\Ac$, which have signal.

\begin{lemma}
  \label{lemma:maximadisconnect}
  Let $S$ be a $p\times p$ correlation matrix, with underlying data
  $Z_1,\dots,Z_p \in \R^n$, where the $Z_j$ are independent and identically
  spherically distributed around the origin.

  Let $E$ be the event that the first $k$ largest off-diagonal elements of $S$ share no
  indices.  Then as $n,p\to\infty$ with $k$ fixed and $\frac{\log p}{n} \to 0$, $\Pr(E)\to
  1$.
\end{lemma}
\begin{proof}
  Note that we can bound this probability by
  \begin{align*}
    \Pr(E^c) &\le \Pr\left(\text{The $k^{th}$ largest $|S_{ij}|$ } <
    \sqrt{\frac{3.5\log p}{n}}\right)\\
    &\qquad + \Pr\left(\exists i,j,k: |S_{ik}| > \sqrt{\frac{3.5\log p}{n}}
    \text{ and } |S_{jk}|> \sqrt{\frac{3.5\log p}{n}}\right).
  \end{align*}
  The first of these can be bounded by Lemma \ref{lemma:chebyshev}, giving 
  \begin{align*}
    \Pr\left(\text{The $k^{th}$ largest $|S_{ij}|$ } <
    \sqrt{\frac{3.5\log p}{n}}\right) &\le
    \frac{1}{\binom{p}{2}^2\bar{F}_n^2(\sqrt{3.5 \log p})}\cdot
    \left(1 + O\left(\frac{\sqrt{\log p}}{p^{1/4}}\right)\right).
  \end{align*}
  By Lemma \ref{lemma:mills}, 
  \begin{align*}
    \bar{F}_n(\sqrt{3.5\log p}) &\ge
    \frac{1}{\sqrt{3.5\log p}}\left(1-\frac{3.5\log
    p}{n}\right)^{n/2-1}\left(1+O\left(\frac{1}{n}+\frac{1}{\log
    p}\right)\right)\\
    &= \frac{1}{\sqrt{3.5\log p}}e^{-1.75\log
    p\left(1+O\left(\frac{1}{n}+\frac{\log p}{n}\right)\right)}\left(1+O\left(\frac{1}{n}+\frac{1}{\log
    p}\right)\right)\\
    &\ge \frac{1}{p^{3/2}\sqrt{\log p}}
  \end{align*}
  where we consider in the last step $n,p$ large enough that the
  $O\left(\frac{1}{n}+\frac{1}{\log p}\right) < 1$ and $O\left(\frac{1}{n} +
  \frac{\log p}{n}\right) < \frac{1}{7}$.
  Then 
  \begin{align*}
    \Pr\left(\text{The $k^{th}$ largest $|S_{ij}|$ } <
    \sqrt{\frac{3.5\log p}{n}}\right) &\le \frac{p^3\log p}{\binom{p}{2}^2}\cdot
    \left(1 + O\left(\frac{\sqrt{\log p}}{p^{1/4}}\right)\right) =
    O\left(\frac{\log p}{p}\right).
  \end{align*}

  For the other term, first note that conditional on $S_{ij}$, the pairs
  $\{S_{ik},S_{jk}\}$ are i.i.d. over all $k\notin \{i,j\}$.  Using this, and
  defining $C = \left\{s: |s| > \sqrt{\frac{3.5\log p}{n}}\right\}$
  \begin{align*}
    &\Pr\left(|S_{ij}|>\sqrt{\frac{3.5\log p}{n}}\text{ and }\exists k: \max\{|S_{ik}|,|S_{jk}|\}>\sqrt{\frac{3.5\log
    p}{n}}\right)\\
    &\qquad= \int_C \Pr\left(\exists k: \max\{|S_{ik}|,|S_{jk}|\}>\sqrt{\frac{3.5\log
    p}{n}}\middle|S_{ij} = s\right)f_n(s)ds\\
    &\qquad= \int_C \left(1-\left(1-\Pr\left(\max\{|S_{ik}|,|S_{jk}|\}>\sqrt{\frac{3.5\log
    p}{n}}\middle|S_{ij} = s\right)\right)^{p-2}\right)f_n(s)ds.
  \end{align*}
  We can then apply a union bound to obtain $\Pr\left(\max\{|S_{ik}|,|S_{jk}|\}>\sqrt{\frac{3.5\log p}{n}}\middle|S_{ij}
  = s\right) \le 2\Pr\left(|S_{ik}|>\sqrt{\frac{3.5\log p}{n}}\middle|
  S_{ij}=s\right)$.  Since the elements of $S$ are pairwise independent, this
  simplifies to $2\Pr\left(|S_{ik}|>\sqrt{\frac{3.5\log p}{n}}\right)$.  Notating
  $\bar{F}_n(\sqrt{3.5\log p}) = \Pr\left(|S_{ik}|>\sqrt{\frac{3.5\log
  p}{n}}\right)$ as before, the bound becomes
  \begin{align*}
    &\Pr\left(|S_{ij}|>\sqrt{\frac{3.5\log p}{n}}\text{ and }\exists k: \max\{|S_{ik}|,|S_{jk}|\}>\sqrt{\frac{3.5\log
    p}{n}}\right)\\
    &\qquad\le \int_C \left(1-\left(1-2\bar{F}_n(\sqrt{3.5\log
    p})\right)^{p-2}\right)f_n(s)ds\\
    &= \left(1-\left(1-2\bar{F}_n(\sqrt{3.5\log
    p})\right)^{p-2}\right)\bar{F}_n(\sqrt{3.5\log p})\\
    &\qquad\le 1-\left(1-\frac{2\sqrt{2}p^{-1.75(1-2/n)}}{\sqrt{3.5\pi\log
        p}}\left(1+O\left(\frac{1}{n}\right)\right)\right)^{p-2}\left(\frac{\sqrt{2}p^{-1.75(1-2/n)}}{\sqrt{3.5\pi\log
        p}}\left(1+O\left(\frac{1}{n}\right)\right)\right)\\
        &\qquad= \frac{4}{3.5\pi p^{2.5(1-2/n)}\log
        p}\left(1+O\left(\frac{1}{n}+\frac{1}{p^{3/4}\log p}\right)\right)\\
  \end{align*}
  where Lemma \ref{lemma:mills} was used in the last step to bound
  $\bar{F}_n(\sqrt{3.5\log p})$.

  Using this bound, the second term from the beginning of this lemma can be bounded,
  \begin{align*}
    &\Pr\left(\exists i,j,k: |S_{ik}| \text{ and } |S_{jk}| >
    \sqrt{\frac{3.5\log p}{n}}\right) \\
    &\qquad\le \sum_{i<j}\Pr\left(|S_{ij}|>\sqrt{\frac{3.5\log
    p}{n}}\text{ and }\exists k: \max\{|S_{ik}|,|S_{jk}|\}>\sqrt{\frac{3.5\log
    p}{n}}\right)\\
    &\qquad\le \sum_{i<j} \frac{4}{3.5\pi p^{2.5(1-2/n)}\log
    p}\left(1+O\left(\frac{1}{n}+\frac{1}{p^{3/4}\log p}\right)\right)\\
    &\qquad\le \frac{2}{3.5\pi p^{1/2 - 5/n}\log
    p}\left(1+O\left(\frac{1}{n}+\frac{1}{p^{3/4}\log p}\right)\right)
     \le o\left(\frac{1}{p^{1/2-\eps}\log p}\right)
  \end{align*}
  for any $\eps > 0$.

  Combining these two bounds, we have $\Pr(E) \ge 1 -
  o\left(\frac{1}{p^{1/2-\eps}\log p}\right)$, so $\Pr(E) \to 1$.
\end{proof}

\begin{lemma}
\label{lemma:signalcasedisconnect}
Consider the setting of Theorem \ref{thm:latersteps}.  For convenience, arrange
the $p$ variables so that the first $a \equiv |\Ac|$ of them are the variables in $\Ac$.  
Let $D$ be the event that the (finite) $k$
largest off-diagonal elements of $S$ outside of $\Ac\times\Ac$ appear in the $(p-a)\times (p-a)$ block
involving only variables $Z_{a+1},\dots,Z_p$.

Then as $n,p\to\infty$ with $\frac{\log p}{n}\to 0$, $\Pr(D) \to 1$.
\end{lemma}
\begin{proof}
This argument follows directly from the proof approach of Lemma
\ref{lemma:maximadisconnect}.  The probability that at least k elements of the
$(p-a)\times(p-a)$ block are above $\sqrt{3.5\log p / n}$ is
$1-O\left(\frac{\log (p-a)}{p-a}\right)$ by the argument provided there.
Furthermore, even with structure in the $a\times a$ block of the true
correlation matrix, the entries $S_{ij}$, $i\in \{1,\dots,a\}$, $j\in
\{a+1,\dots,p\}$ still have the same marginal distribution, and furthermore the
rows $S_{i1},\dots,S_{ia}$ are i.i.d. for different values of $i\in
\{a+1,\dots,p\}$.  Therefore the same union bounding approach as in Lemma
\ref{lemma:maximadisconnect} works, with the $2$ in the union bound being replaced by
$a$, yielding 
\begin{align*}
  \Pr\left(\max_{i\ge a+1,j\le a} |S_{ij}| > \sqrt{\frac{3.5\log p}{n}}\right)
  \le \frac{a\sqrt{2}p^{-.75(1-2/n)}}{\sqrt{3.5\pi \log
    p}}\left(1+O\left(\frac{1}{n}+\frac{1}{p^{3/4}\log p}\right)\right).
\end{align*}
Therefore
\begin{align*}
  \Pr(D) &= 1 - O\left(\frac{p^{-.75(1-2/n)}}{\sqrt{\log p}}+\frac{\log
  (p-a)}{p-a}\right) = 1- o\left(\frac{1}{\sqrt{p}}\right)
\end{align*}
so $\Pr(D) \to 1$.

\end{proof}

\subsection{Bounds relating to the approximate partitions, $A_{\cdots}^*$}

The results in this section are all developed to support Lemma
\ref{lemma:astar}, which shows that the sums of the probabilities of the events $A_{ij}^*$ and $A_{\allij}^*$
have the necessary limits for the theorems presented in the paper.  Lemma
\ref{lemma:int1} is a special case of Conjecture \ref{conjecture} for $k=1$,
where a simple proof holds.  Lemma \ref{lemma:intboundfix} uses this result for $k=1$, and depends on Conjecture 1 for $k>1$.

\begin{lemma}
  \label{lemma:int1}
  For $G_{n,p}(x)$ and $f_n(x)$ as defined previously,  
  \begin{align*}
    \int_0^{\sqrt{3.5\log p}} G_{n,p}(x)f_n(x)dx &\sim o\left(\frac{1}{p^2}\right)
  \end{align*}
\end{lemma}
\begin{proof}
  Applying the Chebyshev bound from Lemma \ref{lemma:chebyshev} yields
  \begin{align*}
    \Pr(\sqrt{n}M_{ij}< x) \le \frac{1}{\binom{p}{2})^2
    \bar{F}_n^2(x)}\left(1+O\left(\frac{\sqrt{\log p}}{p^{1/4}}\right)\right)
  \end{align*}
  for $x \in [0,\sqrt{3.5\log p}]$.  Recall from Lemma \ref{lemma:mills} that
  \begin{align*}
    \frac{x}{x^2+1}\left(1-\frac{x^2}{n}\right)f_n(x) \le \bar{F}_n(x) =
    \Pr(\sqrt{n}S_{ij}>x)
  \end{align*}
  The original integrand can be bounded above by these inequalities, yielding
  \begin{align*}
    \int_0^{\sqrt{3.5\log p}} G_{n,p}(x)f_n(x)dx 
    &\le \left(1+O\left(\frac{\sqrt{\log p}}{p^{1/4}}\right)\right)\int_0^{\sqrt{3.5\log p}} \frac{4/c_n}{p^2(p-1)^2 \frac{x}{x^2+1}
    \left(1-\frac{x^2}{n}\right)^{n/2}} dx,
  \end{align*}
  This integrand can be shown to be bounded above by $\frac{\sqrt{28\pi}\sqrt{\log
  p}}{p^{2.25-1.75\frac{\log p}{n}}}\left(1+O\left(\frac{1}{n}+\frac{1}{\log
  p}\right)\right)$, and so the integral is bounded by
  $\frac{7\sqrt{2\pi} \log p}{p^{2.25-1.75\frac{\log p}{n}}}\left(1+O\left(\frac{1}{n}+\frac{1}{\log
  p}\right)\right)$.  This implies
  that the integral is $o\left(\frac{1}{p^2}\right)$ (as long as $\frac{\log
  p}{n} < 4/49$, which
  must happen eventually since $\frac{\log p}{n} \to 0$).  Furthermore, if
  $\frac{(\log p)^2}{n} \to 0$, 
  then the integral is $O\left(\frac{\log p}{p^{2.25}}\right)$.

\end{proof}

\begin{lemma}
  \label{lemma:int2}
  For $f_n(x), \bar{F}_n(x)$ as previously defined,
  \begin{align*}
    \int_{\sqrt{\left(4-\frac{2}{k+2}\right)\log p}}^{\sqrt{n}} 2p^3\bar{F}_n^{k+1}(x) f_n(x)dx &= o\left(\frac{1}{p^{2k}}\right)
  \end{align*}
\end{lemma}
\begin{proof}
  Using the Mills ratio result from \ref{lemma:mills}, we have
  \begin{align*}
    \int_{\sqrt{\left(4-\frac{2}{k+2}\right)\log p}}^{\sqrt{n}} 2p^3\bar{F}_n^{k+1}(x) f_n(x)dx 
    &\le \int_{\sqrt{\left(4-\frac{2}{k+2}\right)\log p}}^{\sqrt{n}} 
    \frac{2c_n^3p^3}{x^{k+1}}\left(1-\frac{x^2}{n}\right)^{\frac{(k+2)n}{2}-k-3}dx.
  \end{align*}

  Bounding the logarithm of $\left(1-\frac{x^2}{n}\right)^{\frac{(k+2)n}{2}-k-3}$ by its
  tangent at $x=0$, we obtain $\left(1-\frac{x^2}{n}\right)^{\frac{(k+2)n}{2}-k-3}
  \le e^{-\frac{(k+2)}{2}x^2}e^{(k+3)x^2/n}$.  This bound is largest at the left
  endpoint of the integral, $x^2 = \left(4-\frac{2}{k+2}\right)\log p$ (as long as $n\ge 3$).  Substituting this bound into the integral yields
  \begin{align*}
    \int_{\sqrt{\left(4-\frac{2}{k+2}\right)\log p}}^{\sqrt{n}} 2p^3\bar{F}_n^{k+1} f_n(x)dx 
    &\le 2c_n^{k+2} p^{-2k}\left(1+O\left(\frac{\log
    p}{n}\right)\right)\int_{\sqrt{\left(4-\frac{2}{k+2}\right)\log p}}^{\sqrt{n}} \frac{1}{x^2}dx\\
    &\le \frac{2c_n^3}{\sqrt{3}}\frac{1}{p^{2k}(\log p)^{k/2}}\left(1+O\left(\frac{\log
    p}{n}\right)\right)  = o\left(\frac{1}{p^{2k}}\right)
  \end{align*}
\end{proof}

\begin{lemma}
  \label{lemma:intboundfix}
  For $f_n(x),\bar{F}_n(x),G_{n,p}(x)$ as defined previously and $k\ge 1$,
  \begin{align*}
    \int_{0}^{\sqrt{n}} G_{n,p}(x)\bar{F}_n^{k-1}(x)f_n(x) dx \le \int_0^{\sqrt{n}}
    e^{\binom{p}{2}\bar{F}_n(x)} \bar{F}_n(x) f_n(x)dx +
    o\left(\frac{1}{p^{2k}}\right),
  \end{align*}
  where the result depends on Conjecture \ref{conjecture} for $k>1$.
\end{lemma}
\begin{proof}
  For the case $k=1$, Lemma \ref{lemma:int1} implies that 
  $\int_0^{\sqrt{3.5\log p}} G_{n,p}(x) f_n(x) dx  = o\left(\frac{1}{p^2}\right)$,
  and therefore 
  \begin{align*}
    \int_0^{\sqrt{n}} G_{n,p}(x) f_n(x) dx
    &\le \int_{\sqrt{3.5\log p}}^{\sqrt{n}} G_{n,p}(z) f_n(x)dx +
    o\left(\frac{1}{p^{2}}\right).
  \end{align*}
  The pointwise absolute error of the Chen-Stein approximation to $G_{n,p}(x)$
  presented in Lemmas 6.3 and 6.4 of \cite{corrlimit} is bounded by by
  $2p^3\bar{F}_n^2(x)$.  The total error from replacing $G_{n,p}(x)$ in our integral by
  the approximation $e^{-\binom{p}{2}\bar{F}_n(x)}$ is then bounded by
  $\int_{\sqrt{3.5\log p}}^{\sqrt{n}} 2p^3\bar{F}_n^2(x) f_n(x) dx$, which is in turn
  bounded by the result of Lemma \ref{lemma:int2}.  Therefore 
  \begin{align*}
    \int_0^{\sqrt{n}} G_{n,p}(x) f_n(x) dx &\le \int_{\sqrt{3.5\log p}}^{\sqrt{n}}
    e^{-\binom{p}{2}\bar{F}_n(x)} f_n(x)dx +
    o\left(\frac{1}{p^2}\right) \\
    &\le \int_{0}^{\sqrt{n}} e^{-\binom{p}{2}\bar{F}_n(x)} f_n(x)dx +
    o\left(\frac{1}{p^2}\right),
  \end{align*}
  so the lemma holds for $k=1$.

  For $k>1$, we rely on Conjecture \ref{conjecture}.
  As in the $k=1$ case, we split the integral, this time at
  $\sqrt{\left(4-\frac{2}{k+2}\right)\log p}$.  Invoking Conjecture
  \ref{conjecture}, we obtain
  \begin{align*}
    \int_0^{\sqrt{n}} G_{n,p}(x) \bar{F}_n^{k-1}(x)f_n(x) dx
    &\le \int_{\sqrt{\left(4-\frac{2}{k+2}\right)\log p}}^{\sqrt{n}} G_{n,p}(z)
    \bar{F}_n^{k-1}(x)f_n(x)dx +
    o\left(\frac{1}{p^{2k}}\right)
  \end{align*}
  Then the error from replacing $G_{n,p}(x)$ by $e^{-\binom{p}{2}\bar{F}_n(x)}$ is
  bounded by\\
  $\int_{\sqrt{(4-2/(k+2))\log p}}^{\sqrt{n}}2p^3\bar{F}_n^{k+1}(x)f_n(x)dx$, which is 
  $o\left(\frac{1}{p^{2k}}\right)$ by Lemma \ref{lemma:int2}.  Therefore,  
  \begin{align*}
    \int_{0}^{\sqrt{n}} G_{n,p}(x)\bar{F}_n^{k-1}(x)f_n(x) dx \le \int_0^{\sqrt{n}}
    e^{\binom{p}{2}\bar{F}_n(x)} \bar{F}_n^{k-1}(x) f_n(x)dx +
    o\left(\frac{1}{p^{2k}}\right),
  \end{align*}
  so the lemma holds for $k>1$, conditional on Conjecture \ref{conjecture}.
\end{proof}

\begin{lemma}
  \label{lemma:int3}
  Defining $\bar{F}_n(x)$ and $f_n(x)$ as above,
  \begin{align*}
    \int_0^{\sqrt{n}} e^{-\binom{p}{2} \bar{F}_n(x)}f_n(x)dx &\le \frac{1}{\binom{p}{2}}\\
    \int_0^{\sqrt{n}} e^{-\binom{p}{2} \bar{F}_n(x)}\bar{F}_n^k(x) f_n(x)dx &\le
    \frac{k!}{\binom{p}{2}^{k+1}}\\
  \end{align*}
\end{lemma}
\begin{proof}
  To begin, note that $\frac{d}{dx} \bar{F}_n(x) = -f(x)$.  The first integral can
  then be evaluated by substituting $z = \bar{F}_n(x)$ to obtain
  \begin{align*}
    \int_0^{\sqrt{n}} e^{-\binom{p}{2} \bar{F}_n(x)}f_n(x)dx 
    = \int_0^1 e^{-\binom{p}{2} z}dz
    = \frac{1}{\binom{p}{2}} \left(1-e^{-p(p-1)/2}\right)
    \le \frac{1}{\binom{p}{2}}
  \end{align*}
  Similarly, the same substitution can be applied to the second integral,
  followed by integration by parts, to obtain
  \begin{align*}
    \int_0^{\sqrt{n}} e^{-\binom{p}{2} \bar{F}_n(x)}\bar{F}_n^k(x) f_n(x)dx 
    &= \int_0^1 e^{-\binom{p}{2} z} z^k dz
    = -\frac{1}{\binom{p}{2}}e^{-\binom{p}{2}} +
    \frac{k}{\binom{p}{2}}\int_0^1 e^{-\binom{p}{2}z} z^{k-1}dz\\
&\le \frac{k}{\binom{p}{2}}\int_0^1 e^{-\binom{p}{2}z} z^{k-1}dz
  \end{align*}
  Using the result for $k=1$ in the first integral, induction then implies that 
  \begin{align*}
    \int_0^{\sqrt{n}} e^{-\binom{p}{2} \bar{F}_n(x)}\bar{F}_n^k(x) f_n(x)dx 
    &\le \frac{k!}{\binom{p}{2}^{k+1}}
  \end{align*}
\end{proof}

\begin{lemma}
  \label{lemma:astar}

  Consider mutually independent random variables $\tilde{S}_1,\dots,\tilde{S}_k$ with
  distribution $f_n$ and $M$ which is distributed as
  $\max_{i<j}\sqrt{n}|S_{ij}|$ for $S_{ij}$ as in the set up of Theorem
  \ref{thm:globalfirststep}.

  Then 
  \begin{align*}
    \Pr(\tilde{S}_1,\dots,\tilde{S}_k > M) &\le \frac{k!}{\binom{p}{2}^k} +
    o\left(\frac{1}{p^{2k}}\right),
  \end{align*}
  and as a consequence,
  \begin{align*}
    \sum_{i<j} \Pr(A_{ij}^*) \to 1 \qquad \text{and} \qquad 
    \frac{1}{k!}\sum_{\allij} \Pr(A_{\allij}^*) \to 1
  \end{align*}
\end{lemma}

\begin{proof}
  Note that $\Pr(\tilde{S}_1,\dots,\tilde{S}_k > M) = \int_0^{\sqrt{n}} G_{n,p}(x)
  k\bar{F}_n^{k-1}(x)f_n(x) dx$, since the independence of the
  $\tilde{S}_j$ means that the density of their minimum is
  $k\bar{F}_n^{k-1}(x)f_n(x)$.  We then apply Lemma \ref{lemma:intboundfix} to
  obtain the bound 
  \begin{align*}
    \Pr(\tilde{S}_1,\dots,\tilde{S}_k > M) &\le \int_0^{\sqrt{n}}
    e^{-\binom{p}{2}\bar{F}_n(x)} k\bar{F}_n^{k-1}(x) f_n(x) dx +
    o\left(\frac{1}{p^{2k}}\right)
    \le \frac{k!}{\binom{p}{2}^k} + o\left(\frac{1}{p^{2k}}\right)
  \end{align*}
  where the second inequality follows from Lemma \ref{lemma:int3}.   

  \noindent To obtain the limit $\sum_{i<j} A_{ij}^*\to 1$, note that $A_{ij}\subseteq
  A_{ij}^*$ and that $\Pr(A_{ij}) = \frac{2}{p(p-1)}$.  Therefore
  \begin{align*}
    \frac{2}{p(p-1)} \le \Pr(A_{ij}^*) \le \frac{2}{p(p-1)} +
    o\left(\frac{1}{p^2}\right)
  \end{align*}
  and so summing yields,
  \begin{align*}
    1 \le \sum_{i<j} \Pr(A_{ij}^*) \le 1 + o(1).
  \end{align*}

  Similarly, $A_{\allij} \subseteq A_{\allij}^*$, so
  \begin{align*}
    \frac{1}{k!}\sum_{\Ic_k^*} \Pr(A_{\allij}) \le \frac{1}{k!}\sum_{\Ic_k^*} \Pr(A_{\allij}^*) \le
    \frac{1}{k!}\sum_{\Ic_k^*} \frac{k!}{\binom{p}{2}^k} + o(1).
  \end{align*}
  Since both $\frac{1}{k!}\sum_{\Ic_k^*} \Pr(A_{\allij}) \to 1$ and
  $\frac{1}{k!}\sum_{\Ic_k^*}
  \frac{k!}{\binom{p}{2}^k} + o(1) \to 1$,\\
  $\frac{1}{k!}\sum_{\Ic_k^*} \Pr(A_{\allij}^*) \to 1$.   
\end{proof}

\bibliographystyle{imsart-nameyear}
\bibliography{glasso_paper}

\begin{thebibliography}{8}

\bibitem[\protect\citeauthoryear{Baricz}{2008}]{millsproof}
\begin{barticle}[author]
\bauthor{\bsnm{Baricz},~\bfnm{Arpad}\binits{A.}}
(\byear{2008}).
\btitle{Mills' ratio: Monotonicity patterns and functional inequalities}.
\bjournal{Journal of Mathematical Analysis and Applications}
\bvolume{340}
\bpages{1362 - 1370}.
\bdoi{10.1016/j.jmaa.2007.09.063}
\end{barticle}
\endbibitem

\bibitem[\protect\citeauthoryear{Cai and Jiang}{2011}]{corrlimit}
\begin{barticle}[author]
\bauthor{\bsnm{Cai},~\bfnm{Tony}\binits{T.}} \AND
  \bauthor{\bsnm{Jiang},~\bfnm{Tiefeng}\binits{T.}}
(\byear{2011}).
\btitle{Phase transition in limiting distributions of coherence of
  high-dimensional random matrices}.
\bjournal{Journal of Multivariate Analysis}.
\end{barticle}
\endbibitem

\bibitem[\protect\citeauthoryear{Friedman, Hastie and
  Tibshirani}{2008}]{glasso}
\begin{barticle}[author]
\bauthor{\bsnm{Friedman},~\bfnm{Jerome}\binits{J.}},
  \bauthor{\bsnm{Hastie},~\bfnm{Trevor}\binits{T.}} \AND
  \bauthor{\bsnm{Tibshirani},~\bfnm{Rob}\binits{R.}}
(\byear{2008}).
\btitle{glasso: Graphical lasso-estimation of Gaussian graphical models}.
\bjournal{R package version}
\bvolume{1}.
\end{barticle}
\endbibitem

\bibitem[\protect\citeauthoryear{Lockhart
  et~al.}{2013}]{lockhart2013significance}
\begin{barticle}[author]
\bauthor{\bsnm{Lockhart},~\bfnm{Richard}\binits{R.}},
  \bauthor{\bsnm{Taylor},~\bfnm{Jonathan}\binits{J.}},
  \bauthor{\bsnm{Tibshirani},~\bfnm{Ryan}\binits{R.}} \AND
  \bauthor{\bsnm{Tibshirani},~\bfnm{Robert}\binits{R.}}
(\byear{2013}).
\btitle{A significance test for the lasso}.
\bjournal{arXiv preprint arXiv:1301.7161}.
\end{barticle}
\endbibitem

\bibitem[\protect\citeauthoryear{Mazumder and Hastie}{2012a}]{rahulpaper}
\begin{barticle}[author]
\bauthor{\bsnm{Mazumder},~\bfnm{Rahul}\binits{R.}} \AND
  \bauthor{\bsnm{Hastie},~\bfnm{Trevor}\binits{T.}}
(\byear{2012}a).
\btitle{Exact Covariance Thresholding into Connected Components for Large-Scale
  Graphical Lasso}.
\bjournal{J. Mach. Learn. Res.}
\bvolume{13}
\bpages{781--794}.
\end{barticle}
\endbibitem

\bibitem[\protect\citeauthoryear{Mazumder and Hastie}{2012b}]{rahulpaper2}
\begin{barticle}[author]
\bauthor{\bsnm{Mazumder},~\bfnm{Rahul}\binits{R.}} \AND
  \bauthor{\bsnm{Hastie},~\bfnm{Trevor}\binits{T.}}
(\byear{2012}b).
\btitle{The graphical lasso: New insights and alternatives}.
\bjournal{Electronic Journal of Statistics}
\bvolume{6}
\bpages{2125--2149}.
\end{barticle}
\endbibitem

\bibitem[\protect\citeauthoryear{Ravikumar et~al.}{2008}]{ravikumar2008model}
\begin{barticle}[author]
\bauthor{\bsnm{Ravikumar},~\bfnm{Pradeep}\binits{P.}},
  \bauthor{\bsnm{Raskutti},~\bfnm{Garvesh}\binits{G.}},
  \bauthor{\bsnm{Wainwright},~\bfnm{Martin}\binits{M.}} \AND
  \bauthor{\bsnm{Yu},~\bfnm{Bin}\binits{B.}}
(\byear{2008}).
\btitle{Model selection in Gaussian graphical models: High-dimensional
  consistency of l1-regularized MLE}.
\bjournal{Advances in Neural Information Processing Systems (NIPS)}
\bvolume{21}.
\end{barticle}
\endbibitem

\bibitem[\protect\citeauthoryear{Witten, Friedman and
  Simon}{2011}]{witten2011new}
\begin{barticle}[author]
\bauthor{\bsnm{Witten},~\bfnm{Daniela~M}\binits{D.~M.}},
  \bauthor{\bsnm{Friedman},~\bfnm{Jerome~H}\binits{J.~H.}} \AND
  \bauthor{\bsnm{Simon},~\bfnm{Noah}\binits{N.}}
(\byear{2011}).
\btitle{New insights and faster computations for the graphical lasso}.
\bjournal{Journal of Computational and Graphical Statistics}
\bvolume{20}
\bpages{892--900}.
\end{barticle}
\endbibitem

\end{thebibliography}

\end{document}